\documentclass[onefignum,onetabnum]{siamart190516}

\usepackage{enumitem}
\usepackage{cite}

\usepackage{lipsum}
\usepackage{amsfonts}
\usepackage{graphicx}
\usepackage{epstopdf}
\usepackage{adjustbox}
\usepackage{algorithmic}
\ifpdf
  \DeclareGraphicsExtensions{.eps,.pdf,.png,.jpg}
\else
  \DeclareGraphicsExtensions{.eps}
\fi

\makeatletter
\newcommand\notsotiny{\@setfontsize\notsotiny{6}{7}}
\makeatother
\newsiamthm{remark}{Remark}
\newsiamremark{hypothesis}{Hypothesis}
\crefname{hypothesis}{Hypothesis}{Hypotheses}
\newsiamthm{claim}{Claim}
\newsiamthm{assum}{Assumption}
\newsiamthm{scheme}{Scheme}

\makeatletter
\DeclareRobustCommand{\cev}[1]{%
  \mathpalette\do@cev{#1}%
}
\newcommand{\do@cev}[2]{%
  \fix@cev{#1}{+}%
  \reflectbox{$\m@th#1\vec{\reflectbox{$\fix@cev{#1}{-}\m@th#1#2\fix@cev{#1}{+}$}}$}%
  \fix@cev{#1}{-}%
}
\newcommand{\fix@cev}[2]{%
  \ifx#1\displaystyle
    \mkern#23mu
  \else
    \ifx#1\textstyle
      \mkern#23mu
    \else
      \ifx#1\scriptstyle
        \mkern#22mu
      \else
        \mkern#22mu
      \fi
    \fi
  \fi
}

\headers{A probabilistic scheme for nonlocal diffusion equations}{M.~Yang, G.~Zhang, D.~del-Castillo-Negrete, and Y.~Cao}

\title{A probabilistic scheme for semilinear nonlocal diffusion equations with volume constraints\thanks{This manuscript has been authored by UT-Battelle, LLC, under contract DE-AC05-00OR22725 with the US Department of Energy (DOE). The US government retains and the publisher, by accepting the article for publication, acknowledges that the US government retains a nonexclusive, paid-up, irrevocable, worldwide license to publish or reproduce the published form of this manuscript, or allow others to do so, for US government purposes. DOE will provide public access to these results of federally sponsored research in accordance with the DOE Public Access Plan.
}}

\author{M.~Yang\thanks{Fusion Energy Division, Oak Ridge National Laboratory, Oak Ridge, TN
  (\email{yangm@ornl.gov}, \email{delcastillod@ornl.gov}).}
\and G.~Zhang\thanks{Computer Science and Mathematics Division, Oak Ridge National Laboratory, Oak Ridge, TN (\email{zhangg@ornl.gov}).}
\and D.~Del-Castillo-Negrete\footnotemark[2]
\and Y.~Cao\thanks{Department of Mathematics, Auburn University, Auburn, AL (\email{yzc0009@auburn.edu}).}
}

\usepackage{amsopn}

\begin{document}

\maketitle

\begin{abstract}
This work presents a probabilistic scheme for solving semilinear nonlocal diffusion equations with volume constraints and integrable kernels. The nonlocal model of interest is defined by a time-dependent semilinear  partial integro-differential equation (PIDE), in which the integro-differential operator consists of both local convection-diffusion and nonlocal diffusion operators. Our numerical scheme is based on the direct approximation of the nonlinear Feynman-Kac formula that establishes a link between nonlinear PIDEs and stochastic differential equations. 
The exploitation of the Feynman-Kac representation successfully avoids solving dense linear systems arising from nonlocality operators. Compared with existing stochastic approaches, our method can achieve first-order convergence after balancing the temporal and spatial discretization errors, which is a significant improvement of existing probabilistic/stochastic methods for nonlocal diffusion problems. Error analysis of our numerical scheme is established. 
The effectiveness of our approach is shown in two numerical examples. The first example considers a three-dimensional nonlocal diffusion equation to numerically verify the error analysis results. The second example presents a  physics problem motivated by the study of heat transport in magnetically confined fusion plasmas.

\end{abstract}



\section{Introduction}
    
    
    
    

Nonlocal equations appear in many areas of science and engineering. Of particular interest is the study of transport where models involving integrodifferential operators have been proposed to overcome the limitations of  local models based on advection-diffusion equations.
The cornerstone of local transport models is the Fourier-Fick’s law that establishes a linear relationship between the fluxes and the gradients, which leads to local diffusion operators when combined with mass conservation. The widespread use of this type of models is also rooted in connection with the continuous-time random walk driven by Brownian motion, and the characteristic scaling of diffusion processes according to which the mean-squared displacement grows linearly with time. 

However, despite the apparent ubiquity of local diffusive transport, departures from this paradigm have been documented experimentally and  numerically. For example, early work on the study of transport in rotating fluids of interest to geophysical fluid dynamics showed that the presence of coherent structures (e.g., vortices and zonal flows) gives rise to anomalous super-diffusion processes for which the standard  Brownian motion description does not apply \cite{solomon1993observation, del1998asymmetric}. The study of magnetically confined plasmas provides another important example. As in the case of fluids, coherent structures in turbulent plasmas introduce long waiting times and anomalous long displacements known as ``Levy flights" that invalidate the use of local transport models \cite{del2004fractional,sanchez2008nature}. This phenomenology has motivated the development of nonlocal models in which non-diffusive processes in plasmas are described using integrodifferential operators in general and fractional derivatives in particular \cite{van2004probabilistic,del2008fractional}.

Beyond their use in fluids and plasmas, nonlocal equations have found applicability in several other areas of science and engineering including pattern formation \cite{henry2000fractional} and front propagation \cite{mancinelli2002superfast,del2003front} in reaction-nonlocal-diffusion systems, image processing \cite{buades2010image,gilboa2007nonlocal}, option prices in financial markets with jumps \cite{cartea2007fractional}, turbulence \cite{chen2006speculative,gunzburger2018analysis}, groundwater flow and solute transport \cite{benson2000fractional}, peridynamic models of fracture dynamics \cite{silling2000reformulation}, and nonlocal models of epidemic diseases \cite{ahmed2007fractional} among many others. 

Given the vast applications of nonlocal models, it is not surprising that significant efforts have been devoted to the computational aspects of nonlocal equations. 
The two main approaches to the numerical solution of nonlocal diffusion problems can be roughly classified as continuum deterministic methods and particle-based stochastic methods. Deterministic approaches are usually based on extensions of numerical methods for local partial differential equations (PDEs), including finite element method \cite{CHEN20111237,doi:10.1137/13091631X,ainsworth2018towards,doi:10.1137/17M1144696}, finite difference method \cite{lynch2003numerical,meerschaert2004finite,del2006fractional,doi:10.1137/13091631X,DU2016605}, and kernel collocation method \cite{doi:10.1137/19M1277801,TRASK2019151}, among many others. We refer to \cite{d2020numerical} for an overview of some of these methods. 
However, despite the relative success of deterministic methods in local problems, their use in nonlocal problems faces challenges, including the significant increase in computational cost specially in high-dimensional domains. For example, finite element methods need to handle a weak formulation with a sextuple integral in three-dimensional cases, and the large volume of the nonlocal interaction domains dramatically deteriorates the sparsity of the resulting linear system. Moreover, in the presence of  nonlinear forcing terms, an iterative nonlinear solver needs to wrap around the linear solver, making the entire solution process computationally challenging.  
Even though significant efforts have been made to improve the efficiency by exploiting multigrid solvers \cite{Ainsworth2017AspectsOA} or the Toeplitz structure of the linear system \cite{WANG20108095,gao2016fokker} of the fractional Laplacian, the computational inefficiency remains a bottleneck that hinders the progress on the broad applicability of nonlocal diffusion models to  scientific and engineering problems involving high-dimensional irregular domain and nonlinear forcing. 

On the other hand, particle-based stochastic methods approach the numerical solution of nonlocal diffusion problems by exploiting the relation between nonlocal integro-differential operators and general stochastic jump processes 
\cite{cartea2007fluid,2011PhRvE..83a2105B,du2014nonlocal},
a special case of which is the connection between $\alpha$-stable processes  and fractional Laplacians, see for example  \cite{metzler2000random} and references therein. Although this approach does not require assembling and solving dense linear systems and the simulations of a large number of trajectories and can be easily parallelized, it suffers from the slow convergence of random walk models (e.g., $\frac{1}{2}$-order convergence rate with respect to the number of time steps), which requires a very large number of samples to achieve a prescribed accuracy. Additionally, when having an inhomogeneous or even nonlinear forcing term, the nonlocal diffusion equation is no longer the master equation of an underlying jump process for which  random walk methods can be applied.


As an alternative to the above-described deterministic (e.g., finite differences) and stochastic methods (e.g., continuous time random walk models), we present here a new probabilistic scheme for time-dependent nonlocal equations. The specific model under consideration is a semilinear partial integro-differential equation (PIDE) including a local advection-diffusion operator, a nonlocal operator with an integral kernel, and a nonlinear forcing term. The theoretical foundation of our method is rooted in the seminal works on the nonlinear Feynman-Kac theory \cite{Pardoux:1990jua,Peng:1991ku,pardoux1994backward,Pardoux:1992jo,pardoux1998backward,barles1997backward} that establishes a connection between nonlinear PDEs/PIDEs and stochastic processes. 
Similar to our previous use of the Feynman-Kac theory for local problems \cite{yang2021feynman}, the proposed method is a kind of hybrid approach in that it is based on the stochastic representation of the nonlocal kernel but the actual computation is reduced to the continuum deterministic evaluation of integrals bypassing the need of discrete sampling of stochastic trajectories. 

The nonlinear Feynman-Kac theory has been exploited to solve PIDEs in unbounded domains\cite{Bouchard:2008cp,Zhang:2016fb,lejay2014numerical,bouchard2009discrete}, usually achieving $\frac{1}{2}$-order convergence rates with respect to the number of time steps. However, the vast majority of applications of nonlocal models demand the use of finite domains. For example, the application of nonlocal fractional transport models in magnetically confined fusion plasmas requires the regularization of the fractional derivatives to incorporate physically meaningful boundary conditions \cite{del2006fractional}.
The use of volume constraints as a proxy for boundary conditions that might not be defined for a kernel is another approach in the formulation of well-posed nonlocal problems in bounded domains \cite{d2013fractional,du2013nonlocal, du2014nonlocal,du2012analysis}. 
The advantage of the use of volume constraints can be intuitively appreciated in the context of stochastic processes. In the case of local diffusion, the corresponding stochastic process is continuous, and the boundary corresponds to the exit location of the trajectory. However, in a nonlocal problem the underlying stochastic process is a discontinuous jump process and the trajectory can exit the domain without ``touching" the boundary. Adding a volume constrain provides a fix to this problem by identifying the exit of the bounded domain as an arrival to the added volume. 
Here we adopt the volume constraints approach as a natural extension of Dirichlet boundary conditions, and show that our method provides an accurate and efficient numerical technique that improves the convergence rate of existing methods and opens the possibility of applying nonlocal models to high-dimensional practical problems with a particular interest to magnetically confined fusion plasmas.  


In the proposed method the numerical solution of the PIDE's is reduced to the accurate approximation of the expectation value in the Feynman-Kac representation. This task consists of several steps, including discretizing the underlying stochastic process, approximating the nonlinear forcing term, handling the exit time (i.e., the random time that the stochastic process exits the bounded domain), decomposing the entire mathematical expectation into a set of conditional expectations, and picking quadrature weights and abscissa for each conditional expectation. The key algorithmic development and error analysis challenge is the low convergence rate caused by the exit time. To address this problem,
we develop an effective strategy to successfully improve the convergence rate to first-order that is comparable to PDE-based approaches. On the other hand, our method does not require assembling and solving possibly dense linear systems, which significantly improves the overall computational efficiency. Among the main contributions of this work are:
(i) Development of a fully discrete scheme for the semilinear nonlocal diffusion equations with volume constraints and integrable kernels;
(ii) Error estimates of the proposed fully discrete scheme, which demonstrates the first-order convergence with respect to the time step size; 
(iii) Demonstration of our method's performance on 3D semilinear nonlocal diffusion problems in non-trivial domains, and an anisotropic  nonlocal heat transport problem of interest to magnetically confined controlled nuclear fusion plasmas.

The outline of the rest of the paper is as follows. In Section \ref{sec:setting} we formulate the nonlocal volume-constrained problem of interest. The details of the proposed method are described in Section \ref{num_scheme}, and the corresponding error analysis for the fully discrete scheme is studied in Section \ref{sec:err_analy}. Section \ref{sec_ex} is devoted to examples including nonlocal diffusion in four different 3D domains, and nonlocal anisotropic transport in a 3D toroidal domain of interest to controlled nuclear fusion.

\section{Problem setting}\label{sec:setting}
Let $\mathcal{D} \subset \mathbb{R}^d$ denote a bounded open domain and $[0,T]$ with $T>0$ denote a temporal domain. The domain $\mathcal{D}_{\rm v}$ is the interaction domain that is disjoint from $\mathcal{D}$.
The PIDE of interest is a time-dependent semilinear nonlocal volume-constrained diffusion equation, i.e., 
%
\begin{equation}\label{pide}
\begin{aligned}
\dfrac{\partial u}{\partial t}(t,x) - \mathcal{L}[u](t,x) & = f(t,x,u), \quad\quad \forall (t,x) \in (0,T] \times \mathcal{D},\\
u(0,x) & = \phi_0(x), \quad\quad\;\;\;\;\;\forall x\in \mathcal{D}\cup \mathcal{D}_{\rm v},\\[4pt]
u(t,x) & = \phi_{\rm v}(t,x), \quad\quad \;\;\forall (t,x)\in (0,T] \times\mathcal{D}_{\rm v},
\end{aligned}
\end{equation}
where $f(t,x,u)$ is the forcing term that could be a nonlinear function of $u$, $u(0,x) = \phi_0(x)$ is an initial condition, and $u(t,x) = \phi_{\rm v}(t,x)$ is the volume constraint acting on the nonzero domain $\mathcal{D}_{\rm v}$.
%
The volume constraint is a natural extension of the boundary condition for local PDEs \cite{du2014nonlocal,du2012analysis}. The partial integro-differential operator $\mathcal{L}$ in Eq.~\eqref{pide} is defined by
\begin{equation}\label{fpon}
\begin{aligned}
\mathcal{L}[u](t,x) = &
\sum_{i=1}^{d}\frac{\partial}{\partial x_{i}} [B_{i}(t,x) u(t,x)] +
\sum_{i,j=1}^{d}\frac{\partial^{2} }{\partial x_{i} \partial x_{j}}[K_{ij}(t,x)u(t,x)] \\
&+\int_{E}[u(t,x+c(t,x,q)) - u(t,x)]\gamma(q)dq,
\end{aligned}
\end{equation}
where $B = (B_1, \ldots, B_d) \in \mathbb{R}^d$ is the local convection coefficient, 
$K = [K_{ij}]\in \mathbb{R}^{d\times d}$ 
is the local diffusion coefficient satisfying $K = \frac{1}{2}\sigma \sigma^{\top}$ with $\sigma \in \mathbb{R}^{d\times d}$, $c(t,x,q)\in \mathbb{R}^d$ is the jump amplitude,  $E\subset \mathbb{R}^d$ defines the interaction domain for $x\in\mathcal{D}$, and $\gamma(q)$ is the nonlocal kernel. In this work, we assume $\gamma(q)$ is nonnegative and integrable, i.e.,
\begin{equation}\label{cond_lambda}
\gamma(q) \geq 0 \text{ for } q \in E \quad \text{and}\quad \varphi(q) = \frac{\gamma(q)} {\lambda} \text{ with } \lambda =\int_{E} \gamma(q) dq < \infty,
\end{equation}
where $\varphi(q)$ can be viewed as a probability density function and the domain $E$ is bounded. In this case, the nonlocal component in $\mathcal{L}$ corresponds to the compound Possion process. The well-posdeness of the problem in Eq.~\eqref{pide} has been proved in \cite{du2014nonlocal,du2012analysis} under standard assumptions on $B$, $K$, $c$, and $\gamma$. 


It is well known that the nonlocal diffusion problem in Eq.~\eqref{pide} is computationally challenging to solve using standard PDE solvers, especially when $d \geq 3$ and the domain $E$ has a large volume. For example, the classic finite element method needs to handle a weak formulation with a sextuple integral in the case of $d=3$, and the large volume of $E$ will dramatically deteriorate the sparsity of the resulting linear system. Moreover, when having a nonlinear forcing term, an iterative nonlinear solver needs to wrap around the linear solver, making the entire solution process computationally inefficient.  
To circumvent these challenges, we will exploit the connection between the operator $\mathcal{L}$ in Eq.~\eqref{pide} and stochastic jump processes to develop an efficient and accurate probabilistic scheme.

\section{The proposed probabilistic scheme}\label{num_scheme}
In this section, we construct the proposed probabilistic scheme for the nonlocal diffusion problem in Eq.~\eqref{pide}. In Section \ref{Feynman}, we use the Feynman-Kac formula to represent the solution $u$ of the PIDE as a conditional expectation  \cite{pardoux1994backward,Pardoux:1990jua}, which serves as the foundation of our numerical scheme. In Section \ref{sec:full}, we discretize the Feynman-Kac representation to obtain an approximation to $u$. 


\subsection{The Feynman-Kac representation of the PIDE's solution}\label{Feynman}

\subsubsection{The non-divergence form of the PIDE}\label{sec3_1}
The nonlocal diffusion equation in Eq.~\eqref{pide} is given in the divergence form, but the Feynman-Kac formula requires that the integro-differential operator is written in the {non-divergence} form \cite{10.5555/129416}. Thus, we rewrite the PIDE in Eq.~\eqref{pide} in its non-divergence form, i.e., 
%
%
%
\begin{equation}\label{pidedn}
\dfrac{\partial u}{\partial t}(t,x) - {\mathcal{L}}^*[u](t,x) = g(t,x,u),
\end{equation}
where the {non-divergence} form operator $\mathcal{L}^*$ is defined by
\begin{equation}\label{Ls}
\begin{aligned}
\mathcal{L}^*[u](t,x):=&\sum_{i=1}^{d}b_{i}(t,x)\frac{\partial u}{\partial x_{i}}(t,x) + \sum_{i,j=1}^{d}K_{ij}(t,x) \frac{\partial^2 u}{\partial x_{i} x_{j}}(t,x)\\
 &+\int_{E}[u(t,x + c(t,x,q)) -u(t,x)]\gamma(q)dq,
\end{aligned}
\end{equation}
with the new drift coefficients $b_i$ defined by
$
b_{i}(t,x):= B_{i}(t,x) + 2\sum_{j=1}^{d}\frac{\partial K_{ij}}{\partial x_{j}}(t,x),
$ 
and the new forcing term $g$ given by 
\begin{equation}\label{force_g}
g(t,x,u):= f(t,x,u) + \left(\sum_{i=1}^{d}\frac{\partial B_{i}}{\partial x_{i}}(t,x)  + \sum_{i,j=1}^{d}\frac{\partial^2 K_{ij}}{\partial x_{i}\partial x_{j}}(t,x)\right)u(t,x).
\end{equation}
Note that  Eq.~\eqref{pidedn} is exactly the same as Eq.~\eqref{pide}, and in the rest of this section, we use the PIDE in Eq.~\eqref{pidedn} as the target problem to develop our probabilistic scheme. 


\subsubsection{The Feynman-Kac representation}\label{sec:rep}
%
For the purpose of the numerical method to be described in Section \ref{sec:full}, we only need to consider the Feynman-Kac formula within a small time interval. Thus, we first introduce a uniform mesh over the temporal domain $[0,T]$ as follows
\begin{equation}\label{t_mesh}
\mathcal{T} := \{0=t_{0} \le \cdots \le t_{N_t}=T\},
\end{equation}
with $\Delta t= t_{n}-t_{n-1}$, for $1\leq n \leq N_t$. In each small interval $[t_n, t_{n+1}]$, we define a {\em backward} stochastic process that starts from the location $(t_{n+1},x)$ and moves backward from $t_{n+1}$ to $t_n$, i.e.,
\vspace{0.2cm}
%
\begin{equation}\label{sde} 
\cev{X}_{s}^{n+1} = x + \int^{t_{n+1}}_{s}b(t,\cev{X}_{t}^{n+1})d{t}+ \int^{t_{n+1}}_{s}\sigma(t,\cev{X}_{t}^{n+1})dW_{t} + \sum_{k = 1}^{N_{t_{n+1} - s}} c (t,\cev{X}_{t_k+}^{n+1},q_k),
\end{equation}
where $s \in [t_n, t_{n+1}]$, and the almost sure right-hand limit of $\cev{X}_t = \{\cev{X}_t, t\in [t_n, t_{n+1}]\}$ is defined by 
\[
\cev{X}_{t+} = \lim_{s\downarrow t}\cev{X}_s.
\]
Here the coefficients $b, c$ are defined in Eq.~\eqref{Ls}, $\sigma$ results from the definition of diffusion coefficient $K$ in Eq.~\eqref{fpon},
$W_t$ is the Brownian motion with the property that $\mathbb{V}ar(W_t) = dt$, $N_{t_{n+1} - s}$ is the Poisson process following the Poisson probability distribution 
\[
\mathbb{P}(N_{t_{n+1} - s}=k) = (\lambda (t_{n+1}-s))^k \frac{e^{-\lambda (t_{n+1}-s)}}{k!},
\]
with $\lambda$ defined in Eq.~\eqref{cond_lambda}, $t_k$ for $k = 1, \ldots N_{t_{n+1} - s}$ is the instances of time that jumps occur, and $q_k$ follows the probability distribution defined by $\varphi(q)$ in Eq.~\eqref{cond_lambda}. The jump process in Eq.~\eqref{sde} is also called compound Poisson process.




\begin{remark}
The backward stochastic process $\cev{X}_{s}^{n+1}$ in Eq.~\eqref{sde} is defined independently for each time interval $[t_n, t_{n+1}]$ for the convenience in developing the numerical scheme. Thus, there is no continuous filtration from $T$ to the initial time. {We emphasize that process $\cev{X}_{s}^{n+1}$ depends on the starting location $x$, i.e., $\cev{X}_{s}^{n+1,x}$. In what follows, we omit $x$ in the superscript for notational simplicity.}
\end{remark}

We define the exit time of $\cev{X}_{s}^{n+1}$ to describe the volume constraint in Eq.~\eqref{pide} from the probabilistic perspective as follows
\begin{equation}\label{exit}
\begin{aligned}
& \tau_{n} := \sup\{s < t_{n+1}\, \big|\, \cev{X}_{s}^{n+1} \not \in \mathcal{D}, x\in \mathcal{D}\},\\
\end{aligned}
\end{equation}
where $\tau_{n}$ indicates the first instance of time $\cev{X}_{s}^{n+1}$ exits the domain $\mathcal{D}$.
Note that $\cev{X}_{s}^{n+1}$
could exit the domain $\mathcal{D}$ in two ways. The first way is that
$\cev{X}_{s}^{n+1}$ exits the domain through the boundary $\partial \mathcal{D}$; the second way is that $\cev{X}_{s}^{n+1}$
jumps out of the domain $\mathcal{D}$ without touching the boundary $\partial \mathcal{D}$.


It is well known that the operator $\mathcal{L}^*$ in Eq.~\eqref{Ls} is the {infinitesimal} generator of $\cev{X}_{s}^{n+1}$ for $s \in [\tau_{n} \vee t_{n},t_{n+1}]$. Thus, we can derive the Feynman-Kac representation \cite{dynkins, Peng:1990vu} of the PIDEs solution at $t_{n+1}$ as follows:
\begin{equation}\label{ito_for1}
u(t_{n+1},x) =  \mathbb{E}\bigg[u(\tau_{n} \vee t_{n},\cev{X}_{\tau_{n} \vee t_{n}}^{n+1}) +  \int^{t_{n+1}}_{\tau_{n} \vee t_{n}}\Big(\frac{\partial u}{\partial t}-\mathcal{L}^*[u]\Big)(t,\cev{X}_{t}^{n+1})d{t}\bigg],
\end{equation}
 where $\mathbb{E}[\cdot]$ denotes a conditional expectation, and
 $\tau_{n} \vee t_{n} :={\rm max}(\tau_{n}, t_{n})$.
%
%
If $u$ is the unique viscosity solution of the nonlocal diffusion equation in Eq.~\eqref{pide}, then the representation in Eq.~\eqref{ito_for1} can be rewritten as
\begin{equation}\label{eq_utn}
\begin{aligned}
u(t_{n+1},x)
=\mathbb{E}\bigg[u(\tau_{n} \vee t_{n},\cev{X}_{\tau_{n} \vee t_{n}}^{n+1}) +\int^{t_{n+1}}_{\tau_{n} \vee t_{n}}g(t,\cev{X}_{t}^{n+1},u(t,\cev{X}_{t}^{n+1}))d{t} \bigg],
\end{aligned}
\end{equation}
where $g$ is the forcing term defined in Eq.~\eqref{force_g}.

Instead of directly solving the PIDE in Eq.~\eqref{pide}, we intend to approximate the solution $u$ at each time step by discretizing the Feynman-Kac representation in Eq.~\eqref{eq_utn}, which will be described in the next section.

\subsection{The approximation of the Feynman-Kac representation}\label{sec:full}
The approximation of the representation of $u(t_{n+1},x)$ defined in Eq.~\eqref{eq_utn} consists of five tasks: (a) discretization of the time integral in Eq.~\eqref{eq_utn};
(b) numerical treatment of the exit time $\tau_n$; (c) approximation of the backward stochastic process $\cev{X}_{s}^{n+1}$; (d) approximation of the expectation $\mathbb{E}[\cdot]$; and (e) reconstruction of $u(t_{n},x)$ in $\mathcal{D}$. These five tasks will be accomplished in Sections \ref{sec:timedis}, \ref{sec:exit}, \ref{sec:euler}, \ref{sec:exp} and \ref{subsec:full}, respectively. 

To proceed, we extend the solution $u$ from the bounded domain $\mathcal{D}\cup\mathcal{D}_{\rm v}$ to $\mathbb{R}^d$. According to the Whitney extension theorem \cite{hestenes1941extension,10.2307/1989708}, a function of class $\mathcal{C}^m$ on a closed set in $\mathbb{R}^d$ can be extended to the entire $\mathbb{R}^d$ and the extended function is still in the class $\mathcal{C}^m$. The purpose of defining the extension of $u$ is only to ensure the mathematical rigor during the derivation of the proposed numerical scheme in the rest of this section. For example, the stochastic process $\cev{X}_{t_n}^{n+1}$ in Eq.~\eqref{sde} could move to anywhere in $\mathbb{R}^d$, but 
the expectation $\mathbb{E}[u(t_n, \cev{X}_{t_n}^{n+1})]$ is not well defined unless $u$ is extended to $\mathbb{R}^d$.  
However, the final numerical scheme does not use any information of the extension, so we only need the existence of the extension. For simplicity, we use the same notation $u$ to denote its extension in the rest of the paper.


\subsubsection{Temporal discretization}\label{sec:timedis}
We use the implicit Euler scheme to discretize the temporal integral in Eq.~\eqref{eq_utn} and obtain
\begin{equation}\label{eq_utn_approx1}
\begin{aligned}
u(t_{n+1},x)
=&\mathbb{E}\left[u(\tau_{n} \vee t_{n},\cev{X}_{\tau_{n} \vee t_{n}}^{n+1}) +(t_{n+1}-\tau_{n} \vee t_{n})g(t_{n+1},x,u(t_{n+1},x))\right] + R^{n+1}_1,
\end{aligned}
\end{equation}
where the truncation error $R_{n+1}^1$ is defined by
\begin{equation}\label{R1}
   R^{n+1}_1 =\hspace{-0.05cm} \mathbb{E}\hspace{-0.05cm}\left[\int^{t_{n+1}}_{\tau_{n} \vee t_{n}}g(t,\cev{X}_{t}^{n+1},u(t,\cev{X}_{t}^{n+1}))d{t} - (t_{n+1}-\tau_{n} \vee t_{n})g(t_{n+1},x,u(t_{n+1},x))\right].
\end{equation}
Even though other time stepping schemes could also be used here, the implicit Euler scheme has sufficient accuracy and stability to achieve the overall first-order convergence with respect to $\Delta t$.


%
\subsubsection{Treatment of the exit time}\label{sec:exit}
Now we describe how to handle the exit time $\tau_n$ in $\mathbb{E}[u(\tau_{n} \vee t_{n},\cev{X}_{\tau_{n} \vee t_{n}}^{n+1})]$ in Eq.~\eqref{eq_utn_approx1}.
The approximation of a mathematical expectation becomes challenging in the presence of an exit time. The commonly used strategies will lead to a {half-order} convergence rate with respect to $\Delta t$ \cite{Gobet:2000cj,buchmann2003computing}, which will not achieve our objective. Here, we develop a easy to use treatment for the exit time, exclusively designed for the nonlocal problem, which achieves an overall first-order convergence with respect to $\Delta t$.
%
%
To proceed, the expectation 
$\mathbb{E}[u(\tau_{n} \vee t_{n},\cev{X}_{\tau_{n} \vee t_{n}}^{n+1})]$ in Eq.~\eqref{eq_utn_approx1} can be decomposed based on different scenarios of $N_{\Delta t} = N_{t_{n+1}-t_n} $ and $\tau_{n}$ as follows.
\begin{equation}\label{e24}
\begin{aligned}
&  \mathbb{E}\left[u(\tau_{n} \vee t_{n},\cev{X}_{\tau_{n} \vee t_{n}}^{n+1})\right] = I_1 + I_2 + I_3 + I_4 + I_5,\\[2pt]
& I_1 = \mathbb{P}(N_{\Delta t} = 0, \tau_{n} \ge t_{n}) \mathbb{E} \left[ u(\tau_{n}, \cev{X}^{n+1}_{\tau_n})|N_{\Delta t} = 0, \tau_{n} \ge t_{n}\right],\\[2pt]
& I_2 = \mathbb{P}(N_{\Delta t} = 0, \tau_{n} < t_{n})\,\mathbb{E}\left[u(t_n,\cev{X}^{n+1}_{t_n})| N_{\Delta t} = 0, \tau_{n} < t_{n}\right],\\[2pt]
%
%
& I_3 = \mathbb{P}(N_{\Delta t} = 1, \tau_{n} \ge t_{n}) \mathbb{E} \left[  u(\tau_{n},\cev{X}^{n+1}_{\tau_n})|N_{\Delta t} = 1,\tau_{n} \ge t_{n}\right],\\[2pt]
& I_4 = \mathbb{P}(N_{\Delta t} = 1, \tau_{n} < t_{n}) \mathbb{E} \left[ u(t_n,\cev{X}^{n+1}_{t_n})|N_{\Delta t} = 1,\tau_{n} < t_{n}\right],\\[-2pt]
& I_5 = \sum_{k=2}^{\infty}\mathbb{P}(N_{\Delta t}  = k) \mathbb{E} \left[u(\tau_{n}\vee t_{n},\cev{X}^{n+1}_{\tau_n \vee t_n})|N_{\Delta t} =k \right].
\end{aligned}
\end{equation}
%
%
Each term in Eq.~\eqref{e24} is an expectation conditional on an event defined by $N_{\Delta t}$ and $\tau_{n}$.

Next we investigate the terms in Eq.~\eqref{e24}, to determine which are small enough to be neglected in the numerical scheme. For the terms that we need to keep in the final numerical scheme, we want to avoid direct approximation of the exit time. Specific treatment of each term is given as follows.
\vspace{0.2cm}
\begin{itemize}[leftmargin=20pt]\itemsep0.15cm
\item We neglect $I_1$ because  the probability 
$\mathbb{P}(N_{\Delta t} = 0, \tau_{n} \ge t_{n})$ is in the order of $\mathcal{O}((\Delta t)^2)$. To see this, we note that when $N_{\Delta t} = 0$, the motion of $\cev{X}_s^{n+1}$ is driven by the Brownian motion, and $\mathbb{P}(N_{\Delta t} = 0, \tau_{n} \ge t_{n})$ decays rapidly as the starting location $x$ of $\cev{X}_s^{n+1}$ moves further away from the boundary $\partial \mathcal{D}$. In fact, we proved in our previous work \cite{Yang:2018fd} that if $b(t,x)$ and $\sigma(t,x)$ are bounded in $[0,T] \times \mathcal{D}$ and the starting location $x$ of $\cev{X}_s^{n+1}$ satisfies 
\begin{equation}\label{cond_bound}
{\rm dist}(x, \partial \mathcal{D}) \ge \mathcal{O}((\Delta t)^{\frac{1}{2}-\varepsilon}),
\end{equation}
for an arbitrarily small positive number $\varepsilon>0$ with ${\rm dist}(\cdot, \cdot)$ denoting the Euclidean distance, then for sufficiently small $\Delta t$, 
\begin{equation}\label{stoperr}
\mathbb{P}(N_{\Delta t} = 0,\tau_{n} \ge t_{n}) \leq C (\Delta t)^\varepsilon \exp\left(-\frac{1}{(\Delta t)^{2\varepsilon}}\right),
\end{equation}
where the constant $C>0$ is independent of $\Delta t$. The condition in Eq.~\eqref{cond_bound} can be satisfied by properly defining the spatial mesh, which will be discussed in Section \ref{subsec:full}. 
The estimate in Eq.~\eqref{stoperr} allows us to neglect $I_1$ and define it as another truncation error term 
\begin{equation}\label{e100}
 R^{n+1}_2 := I_1.   
\end{equation}

%

\item To analyze $I_2$, we rewrite its definition as 
\begin{equation}
I_2 = \mathbb{P}(N_{\Delta t} = 0) \mathbb{E}\left[u(t_n, \cev{X}_{t_n}^{n+1}) \, \big|\, N_{\Delta t} = 0\right] + R^{n+1}_3,
\end{equation}
where the truncation error $R^{n+1}_3$ is defined by
\begin{equation}\label{e101}
\begin{aligned}
 R_3^{n+1} := & \; \mathbb{P}(N_{\Delta t} = 0, \tau_{n} < t_{n})\,\mathbb{E}\left[u(t_n,\cev{X}^{n+1}_{t_n})| N_{\Delta t} = 0, \tau_{n} < t_{n}\right] \\
 &\;- \mathbb{P}(N_{\Delta t} = 0)\mathbb{E}\left[u(t_n, \cev{X}_{t_n}^{n+1})\, \big|\, N_{\Delta t} = 0\right].
\end{aligned}
\end{equation}
It is easy to see that $\mathbb{P}(N_{\Delta t} = 0, \tau_{n} < t_{n}) \rightarrow \mathbb{P}(N_{\Delta t} = 0)$ as $\mathbb{P}(N_{\Delta t} = 0, \tau_{n} \ge t_{n})$ in Eq.~\eqref{stoperr} goes to zero. So the error $R^{n+1}_3$ will be sufficiently small when Eq.~\eqref{stoperr} holds. The estimate of $R^{n+1}_3$ will be given in Section \ref{sec:R3}.

\item For $I_3$, we introduce the auxiliary variable, $\cev{V}^{n+1}_s$, 
defned as
%
\begin{equation}\label{aux}
\begin{aligned}
   \cev{V}^{n+1}_s := & \;\cev{{X}}^{n+1}_s -  \int^{t_{n+1}}_{s}b(t,\cev{X}_{t}^{n+1})d{t}- \int^{t_{n+1}}_{s}\sigma(t,\cev{X}_{t}^{n+1})dW_{t} \\
   = &\; x + \sum_{k = 1}^{N_{t_{n+1} - s}} c (t,\cev{X}_{t_k}^{n+1},q_k),
\end{aligned}
\end{equation}
which is the truncation of the $\cev{{X}}^{n+1}_s$ increment by only keeping the jump component. Using this variable, we can rewrite $I_3$ as 
\begin{equation}\label{I3}
    I_3 = \mathbb{P}(N_{\Delta t} = 1, \tau_{n} \ge t_{n}) \mathbb{E} \left[  u(t_{n},\cev{V}^{n+1}_{t_n})|N_{\Delta t} = 1,\tau_{n} \ge t_{n}\right] + R_4^{n+1},
\end{equation}
where the truncation error $R_4^{n+1}$ is defined by
\begin{equation}\label{r4}
\begin{aligned}
   R^{n+1}_4 := \mathbb{P}(N_{\Delta t} = 1, \tau_{n} \ge t_{n}) \mathbb{E} \left[  u(\tau_{n},\cev{X}^{n+1}_{\tau_n}) - u(t_{n},\cev{V}^{n+1}_{t_n})\,\big|\,N_{\Delta t} = 1,\tau_{n} \ge t_{n}\right].
\end{aligned}
\end{equation}
As the probability of having one jump (i.e., $N_{\Delta t} = 1$) is in the order of $\mathcal{O}(\Delta t)$, we only need the expectation in Eq.~\eqref{r4} to be on the order of $\mathcal{O}(\Delta t)$ to achieve the desired $\mathcal{O}((\Delta t)^2)$ local error. The estimate of $R^{n+1}_4$ will be given in Section \ref{analy_exit}. 

\item In a similar way, for $I_4$, we use the auxiliary variable $\cev{V}^{n+1}_s$ in Eq.~\eqref{aux}, and write
\begin{equation}\label{I4}
I_4 = \mathbb{P}(N_{\Delta t} = 1, \tau_{n} < t_{n}) \mathbb{E} \left[  u(t_{n},\cev{V}^{n+1}_{t_n})\, \big|\,N_{\Delta t} = 1,\tau_{n} < t_{n}\right] + R_5^{n+1},
\end{equation}
where the truncation error $R_5^{n+1}$ is defined by
\begin{equation}\label{r5}
\begin{aligned}
   R^{n+1}_5 := \mathbb{P}(N_{\Delta t} = 1, \tau_{n} < t_{n}) \mathbb{E} \left[  u(t_{n},\cev{X}^{n+1}_{t_n}) - u(t_{n},\cev{V}^{n+1}_{t_n})\,\big|\,N_{\Delta t} = 1,\tau_{n} < t_{n}\right].
\end{aligned}
\end{equation}
Similar to $R^{n+1}_4$, we will prove $R^{n+1}_5$ is also on the order of $\mathcal{O}((\Delta t)^2)$ in Section \ref{analy_exit}. Moreover, we can take the sum of $I_3$ in Eq.~\eqref{I3} and $I_4$ in Eq.~\eqref{I4} and obtain
\[
I_3 + I_4 = \mathbb{P}(N_{\Delta t} = 1) \mathbb{E} \left[  u(t_{n},\cev{V}^{n+1}_{t_n})\, \big|\,N_{\Delta t} = 1\right].
\]

\item Finally, for $I_5$, the probability of the Poisson process $N_{\Delta t}$ having $k$ jumps within $[t_n,t_{n+1})$ is on the order of $\mathcal{O}((\Delta t)^k)$, we have 
$
I_5 = \mathcal{O}((\Delta t)^2)
$
when $\mathbb{E} [u(\tau_{n}\vee t_{n},\cev{X}^{n+1}_{\tau_n \vee t_n})|N_{\Delta t} =k]$ for $k\ge 2$ is bounded. So we can neglect $I_5$ in the final numerical scheme and define it as another truncation error term  
\begin{equation}\label{eq_R6}
 R_6^{n+1} = I_5.   
\end{equation}
\end{itemize}

Using these estimates, we rewrite Eq.~\eqref{eq_utn_approx1} as
\begin{equation}\label{e44}
\begin{aligned}
 u(t_{n+1},x) =&\; \mathbb{P}(N_{\Delta t} = 0) \mathbb{E}\left[u(t_n, \cev{X}_{t_n}^{n+1})\, \big|\, N_{\Delta t} = 0\right] \\
 &\;+ \mathbb{P}(N_{\Delta t} = 1) \mathbb{E} \left[  u(t_{n},\cev{V}^{n+1}_{t_n})\, \big|\,N_{\Delta t} = 1\right]\\[-0.25cm]
 &\; +\mathbb{E}[(t_{n+1}-\tau_{n} \vee t_{n})]g(t_{n+1},x,u(t_{n+1},x)) + \sum_{i=1}^6 R^{n+1}_i \, .
\end{aligned}
\end{equation}

\subsubsection{Discretization of \texorpdfstring{$\cev{X}_s^{n+1}$}{lg} and \texorpdfstring{$\cev{V}_s^{n+1}$}{lg}}\label{sec:euler}
To achieve an overall first-order convergence with respect to $\Delta t$, we use the standard Euler scheme \cite{Platen:2010eo} to discretize the stochastic processes
$\cev{X}_{s}^{n+1}$ in Eq.~\eqref{sde} and $\cev{V}_{s}^{n+1}$ in Eq.~\eqref{aux}, i.e., 
%
%
\begin{equation}\label{Euler}
\begin{aligned}
& \cev{X}_{t_n}^{n+1} \approx \cev{X}_n^{n+1} := x + b(t_{n+1},x) \, \Delta t + \sigma(t_{n+1},x) \,\Delta W + \sum_{k = 1} ^{N_{\Delta t}\wedge 1}c (t_{n+1},x, q_k),\\
& \cev{V}_{t_n}^{n+1} \approx \cev{V}_n^{n+1} := x + \sum_{k = 1} ^{N_{\Delta t}\wedge 1}c (t_{n+1},x, q_k),
\end{aligned}
\end{equation}
where $\Delta W :=W_{t_{n+1}} -  W_{t_n}$, and 
$N_{\Delta t}\wedge 1:= {\rm min}(N_{\Delta t},1)$. Note that,
to be consistent with the representation in Eq.~\eqref{e44}, 
we only keep up to one Poisson jump in the approximation. Replacing $\cev{X}_{t_n}^{n+1}$ and $\cev{V}_{t_n}^{n+1}$ in Eq.~\eqref{e44} with $\cev{X}_{n}^{n+1}$ and $\cev{V}_{n}^{n+1}$ in Eq.~\eqref{Euler}, we have
\begin{equation}\label{e45}
    \begin{aligned}
 u(t_{n+1},x) =&\; \mathbb{P}(N_{\Delta t} = 0) \mathbb{E}\left[u(t_n, \cev{X}_{n}^{n+1})\, \big|\, N_{\Delta t} = 0\right] \\
 &+\; \mathbb{P}(N_{\Delta t} = 1) \mathbb{E} \left[  u(t_{n},\cev{V}^{n+1}_{n})\, \big|\,N_{\Delta t} = 1\right]\\[-0.25cm]
 &+\; \mathbb{E}[(t_{n+1}-\tau_{n} \vee t_{n})]g(t_{n+1},x,u(t_{n+1},x)) + \sum_{i=1}^7 R^{n+1}_i,
\end{aligned}
\end{equation}
where a new truncation error $R^{n+1}_7$ is introduced as follows
\begin{equation}\label{R7}
\begin{aligned}
    R^{n+1}_7 := &\; \mathbb{P}(N_{\Delta t} = 0) \mathbb{E}\left[u(t_n, \cev{X}_{t_n}^{n+1})-u(t_n, \cev{X}_{n}^{n+1}) \, \big|\, N_{\Delta t} = 0\right]\\
     &+\; \mathbb{P}(N_{\Delta t} = 1) \mathbb{E} \left[  u(t_{n},\cev{V}^{n+1}_{t_n})-u(t_{n},\cev{V}^{n+1}_{n}) \; \big|\; N_{\Delta t} = 1\right].
\end{aligned}
\end{equation}

\subsubsection{Approximation of the conditional expectations}\label{sec:exp} 
Now we develop a quadrature rule to approximate the conditional expectation $\mathbb{E}[\cdot]$ in Eq.~\eqref{e45}. The expectation $\mathbb{E}[u(t_n, \cev{X}_{n}^{n+1})\, \big|\, N_{\Delta t} = 0]$ has no jumps, and only involves Brownian motion. 
Therefore, 
\begin{equation}\label{E0}
\mathbb{E}\left[u(t_n, \cev{X}_{n}^{n+1})\, \big|\, N_{\Delta t} = 0\right]=
\int_{\mathbb{R}^d} u(t_{n},x +b(t_{n+1},x)\Delta t + \sigma(t_{n+1},x)\sqrt{2\Delta t}\xi)\rho(\xi)d\xi,
\end{equation}
where $\xi := (\xi_1, \ldots, \xi_d)$ follows the normal distribution with the probability density 
$\rho(\xi) := \pi^{-d/2}  \exp(-\sum_{\ell=1}^d \xi_\ell^2)$.
We use the tensor-product Gauss-Hermite quadrature rule to approximate this integral, and denote the approximate expectation as 
\begin{equation}\label{E1}
    \widehat{\mathbb{E}}\left[u(t_n, \cev{X}_{n}^{n+1})\, \big|\, N_{\Delta t} = 0\right] := \sum_{m=1}^{M}w_{m}\,u\big(t_n,x+b(t_{n+1},x){\Delta t} + \sigma(t_{n+1},x)\sqrt{2\Delta t}\, e_{m}\big),
\end{equation}
where $\{w_{m}, e_{m}\}_{m=1}^M$ denote the Gauss-Hermite quadrature weights and abscissa\footnote{We use a single index to represent the tensor-product quadrature rule.}. 

The expectation $\mathbb{E} [u(t_{n},\cev{V}^{n+1}_{n}) \; \big|\; N_{\Delta t} = 1]$ only involves the Poisson jumps
\begin{equation}\label{int_ejump}
\mathbb{E}\left[u(t_n, \cev{V}_{n}^{n+1})\, \big|\, N_{\Delta t} = 1\right]=
\int_{E} u(t_{n},x + c(t_{n+1},x,q))\varphi(q) dq,
\end{equation}
where $\varphi(q)$ is defined in Eq.~\eqref{cond_lambda}. In this case, the choice of the quadrature rule is determined by $\varphi(q)$. For example, if $\varphi(q)$ is bounded and has a compact support, we can use a Gauss-Legendre rule or a  Newton-Cotes rule; if $\varphi(q)$ is singular at the origin, e.g., $\varphi(q) = 1/|q|^z$ with $0<z<1$, then we can use a Gauss-Jacobi rule. 
In general, we write the quadrature approximation of the Poisson jump as
%
\begin{equation}\label{E2}
   \widetilde{\mathbb{E}}\left[u(t_n, \cev{V}_{n}^{n+1})\, \big|\, N_{\Delta t} = 1\right] := \sum_{l=1}^{L} v_{l}\, u\big(t_n,x+ c(t_{n+1},x,a_l)\big),
\end{equation}
where $\{v_{l}, a_l\}_{l = 1}^L$ denote the corresponding quadrature weights and abscissa.

Using the above quadrature rules in the approximation of the conditional expectations in  Eq.~\eqref{e45} we can write the solution $u(t_{n+1}, x)$ as
%
\begin{equation}\label{eq_utn_approx_tilde}
 \begin{aligned}
 u(t_{n+1},x) =&\; \mathbb{P}(N_{\Delta t} = 0) \widehat{\mathbb{E}}\left[u(t_n, \cev{X}_{n}^{n+1})\, \big|\, N_{\Delta t} = 0\right] \\
&\; + \mathbb{P}(N_{\Delta t} = 1) \widetilde{\mathbb{E}} \left[  u(t_{n},\cev{V}^{n+1}_{n})\, \big|\,N_{\Delta t} = 1\right]\\[-0.25cm]
 &\;+ \mathbb{E}[(t_{n+1}-\tau_{n} \vee t_{n})]g(t_{n+1},x,u(t_{n+1},x)) + \sum_{i=1}^8 R^{n+1}_i,
\end{aligned}
\end{equation}
where the new truncation error term $R_8^{n+1}$ comes from the quadrature rules, i.e.,
\begin{equation}\label{error_quad}
\begin{aligned}
R_8^{n+1}(x) = & \;\mathbb{P}(N_{\Delta t} = 0) \left\{{\mathbb{E}}\left[u(t_n, \cev{X}_{n}^{n+1})\, \big|\, N_{\Delta t} = 0\right] - \widehat{\mathbb{E}}\left[u(t_n, \cev{X}_{n}^{n+1})\, \big|\, N_{\Delta t} = 0\right]\right\} \\
& \;+  \mathbb{P}(N_{\Delta t} = 1)\left\{ {\mathbb{E}}\left[u(t_n, \cev{V}_{n}^{n+1})\, \big|\, N_{\Delta t} = 1\right] - \widetilde{\mathbb{E}}\left[u(t_n, \cev{V}_{n}^{n+1})\, \big|\, N_{\Delta t} = 1\right]\right\}.
\end{aligned}
\end{equation}

\subsubsection{Spatial approximation}\label{subsec:full}
For spatial discretization, we use the piecewise Lagrange polynomial interpolation on a triangular or  tetrahedral mesh of the closed domain $\overline{\mathcal{D}}$, where the set of interpolation points is denoted by
 \begin{equation}\label{spatial_mesh}
\mathcal{S}:= \{x_j \in \mathcal{\overline{D}}: j = 1,\ldots, J\}
\end{equation}
with $J$ being the total number of degrees of freedom. Note that we use a single index to denote the grid points $x_j$ to simplify the notation. 
In particular, we define the approximation of $u(t_n, x)$ using a $p$-th order Lagrange nodal basis  \cite{liu2013finite} as
\begin{equation}\label{eq_interp_E}
  u^p(t_{n},x) := \sum_{j =1}^J u(t_{n},x_j)\psi_j(x),
\end{equation}
where $\psi_j$ is the nodal basis function associated with the grid point $x_j$, and $u(t_{n},x_j)$ is the nodal value at $x_j$. 
Substituting Eq.~\eqref{eq_interp_E} into Eq.~\eqref{eq_utn_approx_tilde}, we have
\begin{equation}\label{eq_utn_approx_j}
\begin{aligned}
 u(t_{n+1},x) =&\; \mathbb{P}(N_{\Delta t} = 0) \widehat{\mathbb{E}}\left[u^p(t_n, \cev{X}_{n}^{n+1})\, \big|\, N_{\Delta t} = 0\right] \\
 &+\; \mathbb{P}(N_{\Delta t} = 1) \widetilde{\mathbb{E}} \left[  u^p(t_{n},\cev{V}^{n+1}_{n})\, \big|\,N_{\Delta t} = 1\right]\\[-0.25cm]
 &+\; \mathbb{E}[(t_{n+1}-\tau_{n} \vee t_{n})]g(t_{n+1},x,u(t_{n+1},x)) + \sum_{i=1}^9 R^{n+1}_i,
\end{aligned}
\end{equation}
where the term $R_9^{n+1}$ represents the truncation error from the piecewise polynomial interpolation given by
\begin{equation}\label{R9}
\begin{aligned}
    R_9^{n+1} := & \;\mathbb{P}(N_{\Delta t} = 0) \left\{\widehat{\mathbb{E}}\left[u(t_n, \cev{X}_{n}^{n+1})\, \big|\, N_{\Delta t} = 0\right] - \widehat{\mathbb{E}}\left[u^p(t_n, \cev{X}_{n}^{n+1})\, \big|\, N_{\Delta t} = 0\right]\right\} \\
& \;+  \mathbb{P}(N_{\Delta t} = 1)\left\{ \widetilde{\mathbb{E}}\left[u(t_n, \cev{V}_{n}^{n+1})\, \big|\, N_{\Delta t} = 1\right] - \widetilde{\mathbb{E}}\left[u^p(t_n, \cev{V}_{n}^{n+1})\, \big|\, N_{\Delta t} = 1\right]\right\}.
\end{aligned}
\end{equation}

Recall that the estimate in Eq.~\eqref{stoperr} requires the condition in Eq.~\eqref{cond_bound} imposed on the starting location of $\cev{X}_{s}^{n+1}$. This condition is realized by letting the spatial mesh $\mathcal{S}$ satisfy  
\begin{equation}\label{cond_bound1}
{\rm dist}(x_j, \partial \mathcal{D}) \ge \mathcal{O}((\Delta t)^{\frac{1}{2}-\varepsilon}) \;\text{ for }\; x_j \in \mathcal{S} \cap \mathcal{D}.
\end{equation}
In fact, we only need to impose this condition on the layer of grid points close to the boundary $\partial \mathcal{D}$, so that the other interior grid points will also satisfy this condition. In practice, we realize this condition by setting up the mesh $\mathcal{S}$ such that the quadrature points used in Eq.~\eqref{E1} for all interior grid points are inside the domain $\mathcal{D}$, i.e., 
\begin{equation}\label{e54}
    \{x_j+b(t_{n+1},x_j){\Delta t} + \sigma(t_{n+1},x_j)\sqrt{2\Delta t}\, e_{m}, m = 1,\ldots, M, x_j \in \mathcal{S}\cap \mathcal{D}\} \subset \mathcal{D},
\end{equation}
which is easy to achieve when $b$ and $\sigma$ is bounded in $\overline{\mathcal{D}}$.

\subsection{The fully discrete scheme}\label{sec:algorithm}
The fully discrete scheme is defined by neglecting all the truncation errors $R^{n+1}_i$ for $i = 1, \ldots, 9$, and by performing an iterative update from $t_0$ to $t_{N_t}$.

\begin{scheme}\label{scheme1}
Given a temporal and spatial mesh $\mathcal{T} \times \mathcal{S}$, an initial condition 
and a volume constraint,  the approximate solution, $u^{n+1,p}(x)$, for $n = 0, \ldots, N_t-1$, is obtained through the following steps.
\vspace{0.2cm}
\begin{itemize}[leftmargin=20pt]\itemsep0.2cm
    \item Step 1: Generate the quadrature abscissae for grid points $x_j \in \mathcal{S} \cap \mathcal{D}$, 
    \[
    \begin{aligned}
   & x_j+b(t_{n+1},x_j){\Delta t} + \sigma(t_{n+1},x_j)\sqrt{2\Delta t}\, e_{m}, \text{ for } m = 1, \ldots, M, \\
&     x_j+ c(t_{n+1},x_j,a_l) , \text{ for } l = 1, \ldots, L.
    \end{aligned}
    \]
    
    \item Step 2: Evaluate $u^{n,p}(x)$, defined in Eq.~\eqref{eq_interp_E}, at the quadrature abscissae. 
    
    \item Step 3: Compute the approximate expectations 
    \[
    \widehat{\mathbb{E}}\left[u^{n,p}(\cev{X}_{n}^{n+1})\, \big|\, N_{\Delta t} = 0\right] \;\; \text{ and }\;\;  \widetilde{\mathbb{E}}\left[u^{n,p}(\cev{V}_{n}^{n+1})\, \big|\, N_{\Delta t} = 1\right]
    \]
    via the quadrature rules in Eq.~\eqref{E1} and Eq.~\eqref{E2}, respectively.
    
    \item Step 4: Solve a pointwise nonlinear equation 
    \[
    \begin{aligned}
        u^{n+1}_j =&\; \mathbb{P}(N_{\Delta t} = 0) \widehat{\mathbb{E}}\left[u^{n,p}(\cev{X}_{n}^{n+1})\, \big|\, N_{\Delta t} = 0\right] \\
         &+\; \mathbb{P}(N_{\Delta t} = 1) \widetilde{\mathbb{E}} \left[  u^{n,p}(\cev{V}^{n+1}_{n})\, \big|\,N_{\Delta t} = 1\right] + \Delta t\; g(t_{n+1},x_j,u^{n+1}_j)
    \end{aligned}
    \]
    to obtain $u_j^{n+1}$ that is the approximation of the nodal values $u(t_{n+1}, x_j)$ for the interior grid points $x_j \in \mathcal{S} \cap\mathcal{D}$. {Note that the processes $\cev{X}_{n}^{n+1}$ and $\cev{V}^{n+1}_{n}$ start from $(t_{n+1},x_j)$}.
    
    \item Step 5: Construct the interpolant $u^{n+1,p}(x)$ via Eq.~\eqref{eq_interp_E} using the nodal value estimations $\{u^{n+1}_j\}_{j=1}^J$.
\end{itemize}
\end{scheme}

\subsubsection{Discussion on features of Scheme \ref{scheme1}}
Here we discuss the efficiency and stability properties of the proposed Scheme \ref{scheme1} in comparison with existing PDE approaches. The accuracy of the method will be analyzed in Section \ref{sec:err_analy}.

\paragraph{Efficiency} The nonlinear Feynman-Kac representation in Eq.~\eqref{eq_utn} changes the entire solution paradigm to a probabilistic setting by describing the nonlocality of the operator $\mathcal{L}$ in Eq.~\eqref{pide} using stochastic processes. Our numerical scheme addresses two major bottleneck of existing PDE approaches, e.g., finite element methods. First, the weak formulation of finite element methods for the PIDE in Eq.~\eqref{pide} involves $2d$-dimensional integrals, which makes it challenging to design an accurate quadrature rule for those integrals, especially in high-dimensional spaces (e.g., 3D). Second, when the interaction domain $E$ in Eq.~\eqref{pide} is large, the standard finite element discretization will result in a non-sparse linear system, which poses a significant challenge for linear solvers. Moreover, when the forcing term $f(t,x,u)$ is a nonlinear function of $u$, another layer of the iterative nonlinear solver is needed. In contrast, our scheme does not require solving any linear system, and the nonlinear equation for each grid point $x_j$, i.e., Step 4 in Scheme \ref{scheme1} can be solved independently. This feature makes it straightforward to develop a parallel implementation of the proposed method. 

\paragraph{Stability} The implicit Euler scheme used in Section \ref{sec:timedis} ensures absolute stability for the discretization of the temporal integral. The discretization of the stochastic processes in Section \ref{sec:euler} only proceeds within $[t_n, t_{n+1}]$ for computing the quadrature abscissa used in Section \ref{sec:exp}. In other words, we re-initialize the stochastic processes from the grid points $x_j$ at each time step $t_n$, so our scheme does not have the numerical instability problem associated with the explicit Euler scheme. Moreover, the nonlinear Feynman-Kac formula converts the integro-differential operator to an expectation form, so that our scheme does not require the Courant-Friedrichs-Lewy-type condition imposed on the spatial and temporal mesh sizes.

\section{Error estimates}\label{sec:err_analy}
In this section we present the error analysis of Scheme \ref{scheme1} in the one-dimensional case $(d=1)$. The analysis can be extended to multi-dimensional cases without essential difficulties. For simplicity, we assume both $\mathcal{T}$ and $\mathcal{S}$ are uniform meshes with mesh sizes $\Delta t$ and $\Delta x$, respectively.
We use the piecewise cubic Lagrange interpolation ($p=3$) for the spatial approximation in Eq.~\eqref{eq_interp_E}, and use the trapezoidal quadrature rule to approximate the exepectation in Eq.~\eqref{E2}.

Even though the implementation of Scheme \ref{scheme1} only requires Lipschitz continuity on $b$, $\sigma$, $c$ and $g$ in Eq.~\eqref{pidedn}, we need to impose a stronger regularity condition to prove the first-order convergence with respect to $\Delta t$. 
To proceed, we first introduce the following notation:
\begin{equation}\label{regu_nota}
    \begin{aligned}
    &\mathcal{C}_b^{k_1,\ldots,k_J}(D_1 \times \cdots \times D_J)\\
 :=& \Bigg\{ \phi: \prod_{j = 1}^{J} D_j \rightarrow \mathbb{R} \,\Big |\,\frac{\partial^{\alpha_1}\cdots \partial^{\alpha_J} \phi}{\partial^{\alpha_1}x_1 \cdots \partial^{\alpha_J}x_J} \text{is bounded and continuous} \\
&  \text{\quad for } \alpha_j \leq k_j \text{, } j = 1,\ldots ,J \text{ where, } (\alpha_1, \cdots, \alpha_J) \in \mathbb{N}^J\Bigg\},
    \end{aligned}
\end{equation}
where $J \in \mathbb{N}^+$, and $D_1 \times \cdots \times D_J \subset \mathbb{R}^J$. Using the notation in Eq.~\eqref{regu_nota}, we impose the following assumption on the coefficients and the solution of the PIDE in Eq.~\eqref{pide}, where the rationale of the assumption can be justified by the theoretical analysis on the regularity of PIDEs (e.g., \cite{caffarelli2009regularity}).
\begin{assum} \label{lemma_coeff}
We assume the nonlocal kernel $\gamma$ satisfies the condition in Eq.~\eqref{cond_lambda}, the functions $B, K, c, f, \phi_0, \phi_{\rm v}$ in Eq.~\eqref{pide} satisfy $f \in \mathcal{C}_{b}^{2,4,4}([0,T] \times \mathcal{D} \times \mathbb{R})$, $B \in \mathcal{C}_{b}^{(2,5)}([0,T] \times \mathcal{D})$, $K \in \mathcal{C}_{b}^{2,6}([0,T] \times \mathcal{D})$, $c \in \mathcal{C}_b^{2,4,4}([0,T] \times \mathcal{D} \times E)$, $\phi \in \mathcal{C}_b^{4+\alpha}$ with $\alpha \in (0,1)$, $\phi_{\rm v} \in \mathcal{C}_b^{2,4}((0,T]\times \mathcal{D}_{\rm v})$, and the PIDE in Eq.~\eqref{pide} admits a classical solution $u\in \mathcal{C}_b^{2,4}([0,T]\times \mathcal{D})$. 
\end{assum}
Under Assumption \ref{lemma_coeff}, we have the regularities of functions $b$, $g$ in Eq.~\eqref{pidedn} and $\sigma$ in Eq.~\eqref{sde} as $b \in \mathcal{C}_{b}^{2,4}([0,T] \times \mathcal{D})$, $\sigma \in \mathcal{C}_{b}^{2,4}([0,T] \times \mathcal{D})$ and $g \in \mathcal{C}_{b}^{2,4,4}([0,T] \times \mathcal{D} \times \mathbb{R})$. Such regularity is sufficient for the following error analysis.

\subsection{Upper bounds for the truncation errors}\label{sec:residual}
In this section, we estimate all truncation errors $R^{n+1}_{i}$, for $i = 1,\ldots,9$, generated in the discretization process (see Section \ref{sec:full}). These estimates will play a key role in the error estimate of the approximate solution.

\subsubsection{The estimates of \texorpdfstring{$R^{n+1}_{1}$}{lg} and \texorpdfstring{$R^{n+1}_{7}$}{lg}}
We estimate the truncation errors $R^{n+1}_1$ and $R^{n+1}_7$ from the discretizations of the temporal integral performed in Section \ref{sec:timedis} and the approximation of the discretization of \texorpdfstring{$\cev{X}_s^{n+1}$}{lg} and \texorpdfstring{$\cev{V}_s^{n+1}$}{lg} performed in Section \ref{sec:euler}. Specifically, we have the following lemma.
%
\begin{lemma}\label{lem_R1}
Under Assumption \ref{lemma_coeff}, the errors $R^{n+1}_1$ defined in Eq.~\eqref{R1} and $R^{n+1}_7$ defined in Eq.~\eqref{R7} satisfy
\begin{equation}
    |R^{n+1}_1| \leq C(\Delta t)^2, \quad |R^{n+1}_7| \leq C(\Delta t)^2,
\end{equation}
where $C$ is a constant independent of $\Delta t$.
\end{lemma}
\begin{proof}
Under Assumption \ref{lemma_coeff}, the forward Euler method defined by Eq.~\eqref{Euler} achieves first-order convergence in the weak sense \cite{mikulevivcius1988time,Platen:2010eo}. Specifically, when $b,\sigma \in \mathcal{C}_b^{2,4}([0,T]\times \mathcal{D}), c \in \mathcal{C}_b^{2,4,4}([0,T]\times \mathcal{D}\times E)$, for any $u \in \mathcal{C}_b^{2,4}([0,T]\times \mathcal{D})$, it holds (Theorem 3.3 in \cite{mikulevivcius1988time})
\[
\big|\mathbb{E}[u(t_n, \cev{X}_{t_n}^{n+1})-u(t_n, \cev{X}_{n}^{n+1})]\big| \le C(\Delta t)^2.
\]
Then we have
\begin{equation}\label{analy_R7}
\begin{aligned}
    |R^{n+1}_7| = &\; \mathbb{P}(N_{\Delta t} = 0) \left|\mathbb{E}\left[u(t_n, \cev{X}_{t_n}^{n+1})-u(t_n, \cev{X}_{n}^{n+1}) \, \big|\, N_{\Delta t} = 0\right]\right|\\
    &+ \mathbb{P}(N_{\Delta t} = 1) \left|\mathbb{E} \left[  u(t_{n},\cev{V}^{n+1}_{t_n})-u(t_{n},\cev{V}^{n+1}_{n}) \; \big|\; N_{\Delta t} = 1\right]\right|
    \le
    C(\Delta t)^2,
\end{aligned}
\end{equation}
where constant $C$ depends on the upper bound of $u$.

For notational simplicity, we define $G(t,\cev{X}_t^{n+1}):=g(t,\cev{X}^{n+1}_t,u(t,\cev{X}_t^{n+1}))$, and 
%
%
the differential operators $L^0$, $L^{-1}$ and $L^1$ as
\begin{equation}
    \begin{aligned}
    L^0G(t,\cev{X}_t^{n+1}):= &\;-\frac{\partial G}{\partial {t}}(t,\cev{X}_t^{n+1}) + b\frac{\partial G}{\partial x}(t,\cev{X}_t^{n+1}) + \frac{1}{2}\sigma^2\frac{\partial^2 G}{\partial x^2}(t,\cev{X}_t^{n+1}) \\
    &\;+ \int_E\left[G(t,\cev{X}_{t+}^{n+1} + c(t,\cev{X}_{t+}^{n+1},q)) - G(t,\cev{X}_{t+}^{n+1}) \right]\lambda(dq),\\
    L^1G(t,\cev{X}_t^{n+1}):=&\;\sigma\frac{\partial G}{\partial x}(t,\cev{X}_t^{n+1}),\\
    L^{-1}G(t,\cev{X}_t^{n+1}):=&\;G(t,\cev{X}_{t+}^{n+1} + c(t,\cev{X}_{t+}^{n+1},q)) - G(t,\cev{X}_{t+}^{n+1}).
    \end{aligned}
\end{equation}
Also, we define 
${\mu}(dq,dt)$ as the Poisson random measure of the Poisson process defined in Eq.~\eqref{sde}. The compensator of ${\mu}$ and the resulting compensated Poisson random measure are defined by $\lambda(dq)d{t}$ and ${\nu}(dq,dt) = {\mu}(dq,dt) - \lambda(dq)d{t}$.
Based on the SDE defined in Eq.~\eqref{sde}, the integral form of the It\^{o} formula of $G(s,\cev{X}_s^{n+1})$, for $\tau_{n} \vee t_{n} \le s \le t_{n+1}$, under the condition $\cev{X}_{t_{n+1}} = x_j$, is given by
\begin{equation}
\begin{aligned}
G(s,\cev{X}_{s}^{n+1}) =&\; G(t_{n+1},x_j) + \int^{t_{n+1}}_s L^0G(t,\cev{X}_t^{n+1}) d{t} + \int^{t_{n+1}}_s L^1G(t,\cev{X}_t^{n+1})dW_t \\
&\;+\int^{t_{n+1}}_s\int_E L^{-1}G(t,\cev{X}_t^{n+1}){\nu}(dq,dt),
\end{aligned}
\end{equation}
where ${\nu}$ is the compensated Poisson measure. 
Thus, substituting the above formula into $\mathbb{E}[\int^{t_{n+1}}_{\tau_{n} \vee t_{n}}G(s,\cev{X}_s^{n+1})d{s}]$, we obtain
\begin{equation}
\begin{aligned}
&\mathbb{E}\Big[\int^{t_{n+1}}_{\tau_{n} \vee t_{n}}G(s,\cev{X}_s^{n+1})d{s}\Big] \\
=&\;\mathbb{E} \Big[\int^{t_{n+1}}_{\tau_{n} \vee t_{n}} \Big[ G(t_{n+1},x_j) + \int^{t_{n+1}}_s L^0G(t,\cev{X}_t^{n+1}) d{t} + \int^{t_{n+1}}_s L^1G(t,\cev{X}_t^{n+1})dW_t \\
&\;+\int^{t_{n+1}}_s\int_E L^{-1}G(t,\cev{X}_t^{n+1}){\nu}(dq,dt)\Big]d{s}\Big]
\end{aligned}
\end{equation}
Due to the martingale property of the Brownian motion and compensated Poisson process, the last two terms of the above equation equal zero. Hence we have
\begin{equation}
\begin{aligned}
\mathbb{E}\Big[\int^{t_{n+1}}_{\tau_{n} \vee t_{n}}G(s,\cev{X}_s^{n+1})d{s}\Big] =&\; G(t_{n+1},x_j)\mathbb{E}[(t_{n+1} - \tau_{n} \vee t_{n})] \\
&+\; \mathbb{E} \Big[\int^{t_{n+1}}_{\tau_{n} \vee t_{n}} \int^{t_{n+1}}_s L^0G(t,\cev{X}_t^{n+1})d{t}d{s}\Big].
\end{aligned}
\end{equation}
Therefore,
\begin{equation}
\small
\begin{aligned}
|R^{n+1}_1| =&\; \left|\mathbb{E}\hspace{-0.05cm}\left[\int^{t_{n+1}}_{\tau_{n} \vee t_{n}}g(t,\cev{X}_{t}^{n+1},u(t,\cev{X}_{t}^{n+1}))d{t} - (t_{n+1}-\tau_{n} \vee t_{n})g(t_{n+1},x,u(t_{n+1},x))\right]\right|\\
\le&\;\sup_{[t_n,t_{n+1}]}\mathbb{E}[|L^0G(t,\cev{X}_t^{n+1})|](\Delta t)^2 \le C(\Delta t)^2,
\end{aligned}
\end{equation}
where constant $C$ depends on upper bounds of $b$, $\sigma$, $g$ and their derivatives.
\end{proof}

\subsubsection{The estimates of \texorpdfstring{$R^{n+1}_{2}$}{lg} and \texorpdfstring{$R^{n+1}_{3}$}{lg}}\label{sec:R3}
The truncation errors $R^{n+1}_2$ and $R^{n+1}_3$ are constructed when we estimate the probability $\mathbb{P}(N_{\Delta t} = 0, \tau_{n} \ge t_{n})$ in Section \ref{sec:exit}. As discussed, when the starting point $x_j \in \mathcal{S}$ is far from the boundary $\partial \mathcal{D}$, i.e., $x_j$ satisfies condition in Eq.~\eqref{cond_bound1}, the errors $R^{n+1}_2$ and $R^{n+1}_3$ are of order $\mathcal{O}((\Delta t)^2)$ so they can be neglected. The statement is rigorously proved in the following lemma.

\begin{lemma}\label{lem_R23}
If the spatial mesh $\mathcal{S}$ satisfies the condition in Eq.~\eqref{cond_bound1}, then the errors $R^{n+1}_2$ and $R^{n+1}_3$ are bounded by
\begin{equation}
|R^{n+1}_2| \leq C(\Delta t)^2, \quad |R^{n+1}_3| \leq C(\Delta t)^2, 
\end{equation}
where the constant $C>0$ is independent of $\Delta t$.
\end{lemma}
\begin{proof}
When $x_j\in \mathcal{S}$ satisfies the condition in Eq.~\eqref{cond_bound1}, we can exploit the inequality in Eq.~\eqref{stoperr} to derive that for any positive number $\varepsilon >0$
\begin{equation}
\begin{aligned}
|R^{n+1}_2| \le &\; \mathbb{P}(N_{\Delta t} = 0, \tau_{n} \ge t_{n}) \left|\mathbb{E} \left[ u(\tau_{n}, \cev{X}^{n+1}_{\tau_n})|N_{\Delta t} = 0, \tau_{n} \ge t_{n}\right]\right|\\
\leq &\;C (\Delta t)^\varepsilon \exp\left(-\frac{1}{(\Delta t)^{2\varepsilon}}\right) < C(\Delta t)^2,
\end{aligned}
\end{equation}
for sufficiently small $\Delta t$, where the constant $C$ depends on the upper bound of function $u$. We can have similar derivation for $R^{n+1}_3$, i.e.,
\begin{equation}
\begin{aligned}
|R_3^{n+1}| \leq& \; \mathbb{P}(N_{\Delta t} = 0, \tau_{n} \ge t_{n})\,\big|\mathbb{E}[u(t_n,\cev{X}^{n+1}_{t_n})| N_{\Delta t} = 0, \tau_{n} \ge t_{n}] \big|\\
\leq & \;C (\Delta t)^\varepsilon \exp\left(-\frac{1}{(\Delta t)^{2\varepsilon}}\right) \le C(\Delta t)^2.
\end{aligned}
\end{equation}
For sufficiently small $\Delta t$. The proof is completed.
\end{proof}

\subsubsection{The estimates of \texorpdfstring{$R^{n+1}_{4}$}{lg} and \texorpdfstring{$R^{n+1}_{5}$}{lg}}\label{analy_exit}
The truncation errors $R^{n+1}_4$ and $R^{n+1}_5$ were respectively defined by Eqs.~\eqref{r4} and.~\eqref{r5} when handling the exit time $\tau_n$ with one Poisson jump within $[t_n,t_{n+1})$.
We have the following estimates. 
\begin{lemma}\label{err_bound}
If the mesh $\mathcal{S}$ in Eq.~\eqref{spatial_mesh} satisfies the condition in Eq.~\eqref{cond_bound1}, then, for any grid point $x_j \in \mathcal{S}$, 
the errors $R^{n+1}_4$ defined in Eq.~\eqref{r4} and $R^{n+1}_5$ defined in Eq.~\eqref{r5} can be bounded by
\begin{equation}
  |R^{n+1}_4| \leq C(\Delta t)^2, \quad |R^{n+1}_5| \leq C(\Delta t)^2,
\end{equation}
where constant C is independent with $\Delta t$.
\end{lemma}

\begin{proof}
We first estimate the truncation error $R^{n+1}_4$. To proceed, we define the time instant when jump occurs by $t_{\rm jump}$, and we have $t_n \le t_{\rm jump} \leq t_{n+1}$. The scenarios of $\cev{X}^{n+1}_t$ exiting the domain $\mathcal{D}$ can be categorized into three cases, i.e.,
\vspace{0.2cm}
\begin{enumerate}[leftmargin=30pt]\itemsep0.1cm
    \item $t_n \le t_{\rm jump} <\tau_n < t_{n+1}$, i.e., $\cev{X}^{n+1}_t$ exits $\mathcal{D}$ before the jump;
    \item $t_n \le t_{\rm jump} =\tau_n \le t_{n+1}$, i.e., $\cev{X}^{n+1}_t$ exits $\mathcal{D}$ due to the jump;
    \item $t_n \le \tau_n < t_{\rm jump} \le t_{n+1}$, i.e., $\cev{X}^{n+1}_t$ exits $\mathcal{D}$ after the jump.
\end{enumerate}

In the first case, we learn from Lemma \ref{lem_R23} that when $\cev{X}^{n+1}_t$ starting from a grid point on $\mathcal{S}$, the probability of  $\cev{X}^{n+1}_t$ exiting $\mathcal{D}$ within $[t_n, t_{n+1}]$ without a Poisson jump is very small. In fact, for any $\varepsilon>0$, we have
\begin{equation}\label{eq_r4_1}
    \begin{aligned}
 &\mathbb{P}(N_{\Delta t} = 1,\tau_{n} >t_{\rm jump} \ge t_{n}) \left|\mathbb{E}\left[  u(\tau_{n},\cev{X}^{n+1}_{\tau_n}) - u(t_{n},\cev{V}^{n+1}_{t_n})\,\big|\,N_{\Delta t} = 1\right]\right|\\
& \leq C\, \mathbb{P}(N_{t_{n+1}-\tau_n} = 0,\tau_{n} \ge t_{n}) \leq C (\Delta t)^\varepsilon \exp\left(-\frac{1}{(\Delta t)^{2\varepsilon}}\right) \le C(\Delta t)^2,   
    \end{aligned}
\end{equation}
for sufficiently small $\Delta t$, where the constant $C>0$ only depends on $u$.

In both second and third cases, we have
\begin{equation}
\begin{aligned}
\cev{{X}}^{n+1}_{\tau_n} = \cev{V}^{n+1}_{\tau_n} + \int^{t_{n+1}}_{\tau_n}b(t,\cev{X}_{t}^{n+1})d{t} + \int^{t_{n+1}}_{\tau_n}\sigma(t,\cev{X}_{t}^{n+1})dW_{t} .
\end{aligned}
\end{equation}
Because the compound Poisson process has the property $\cev{V}^{n+1}_{t_n} = \cev{V}^{n+1}_{\tau_n}$, we apply the It\^o formula to $u(\tau_n,\cev{{X}}^{n+1}_{\tau_n})$ at point $(t_n,\cev{V}^{n+1}_{t_n})$ and obtain
\begin{equation}\label{eq_ito}
\begin{aligned}
u(\tau_n,\cev{{X}}^{n+1}_{\tau_n}) =&\; u(t_n,\cev{V}^{n+1}_{t_n}) + \int_{t_n}^{\tau_n}\frac{\partial u}{\partial {t}} d{t} + \int^{t_{n+1}}_{\tau_n}\bigg( b\frac{\partial u}{\partial x} (t,\cev{X}_{t}^{n+1})
+ \frac{\sigma^2}{2}\frac{\partial^2 u}{\partial x^2}\bigg)d{t}\\
&\;+ \int_{\tau_n}^{t_{n+1}}\sigma(t,\cev{X}_{t}^{n+1}) dW_t + o(\Delta t).
\end{aligned}
\end{equation}
Substituting Eq.~\eqref{eq_ito} into Eq.~\eqref{r4}, we have 
\begin{equation}\label{eq_r4_2}
\begin{aligned}
&\mathbb{P}(N_{\Delta t} = 1,t_{\rm jump}\ge \tau_{n} \ge t_{n}) \left|\mathbb{E} \left[  u(\tau_{n},\cev{X}^{n+1}_{\tau_n}) - u(t_{n},\cev{V}^{n+1}_{t_n})\,\big|\,N_{\Delta t} = 1\right]\right|\\
&\leq  C\Delta t \Big|\mathbb{E}\Big[\int_{\tau_n}^{t_{n+1}}\sigma(t,\cev{X}_{t}^{n+1}) dW_t\Big]\Big| + C(\Delta t)^2.
\end{aligned}
\end{equation}
Due to the martingale property of the Brownian motion,
we have
\begin{equation}\label{eq_mart}
\mathbb{E}\Big[\int_{\tau_n}^{t_{n+1}}\sigma(t,\cev{X}_{t}^{n+1}) dW_t\Big] = 0.
\end{equation}
Combining Eq.~\eqref{eq_r4_1} and Eq.~\eqref{eq_r4_2}, we have 
\[
|R^{n+1}_4| \leq C(\Delta t)^2. 
\]
For the error $R^{n+1}_5$, we apply the It\^o formula for $u(t_n,\cev{{X}}^{n+1}_{t_n})$ at point $(t_n,\cev{V}^{n+1}_{t_n})$. Following the same procedure in Eqs.~\eqref{eq_ito}-\eqref{eq_mart}, we obtain
\[
|R^{n+1}_5| \leq C(\Delta t)^2. 
\]
\end{proof}

\subsubsection{The estimate of \texorpdfstring{$R^{n+1}_{6}$}{lg}}
The truncation error $R^{n+1}_6$ is defined when we neglect the case of the Poisson process having $k\ge 2$ jumps. Specifically, the probability of the Poisson process $N_{\Delta t}$ having $k\ge 2$ jumps within $[t_n,t_{n+1})$ is of the order  $\mathcal{O}((\lambda\Delta t)^2)$, where the intensity $\lambda$ is assumed to be bounded in Eq.~\eqref{cond_lambda}.
Hence the error $R^{n+1}_6$ in Eq.~\eqref{eq_R6} has the bound
\begin{equation}\label{R_2}
\begin{aligned}
 |R^{n+1}_6| \le \sum_{k=2}^{\infty}\mathbb{P}(N_{\Delta t}  = k) \left|\mathbb{E} \left[u(\tau_{n}\vee t_{n},\cev{X}^{n+1}_{\tau_n \vee t_n})|N_{\Delta t} =k \right]\right| \leq C(\Delta t)^2,
\end{aligned}
\end{equation}
where the constant $C$ depends on $u$ and $\lambda$.

\subsubsection{The estimate of  \texorpdfstring{$R^{n+1}_{8}$}{lg}}\label{sec:R45}
We analyze the quadrature rule error $R^{n+1}_8$ defined in Eq.~\eqref{error_quad} in the case that the Gauss-Hermite quadrature rule is used to approximate the integral with respect to the Brownian motion, and the trapezoidal rule is used to approximate the integral with respect to the jump. 
%
%

Let $M$ denote the number of Gauss-Hermite quadrature points in each dimension. If $u(t,\cdot)$ is sufficiently smooth, i.e., $\partial^{2M} u/\partial \xi^{2M}$ is bounded, then the Hermite quadrature error is bounded by \cite{2013JSV...332.4403B,yang2021feynman}
\begin{equation}\label{GH_quad}
\left|{\mathbb{E}}\left[u(t_{n},\cev{X}_n^{n+1})\,|\,N_{\Delta t} = 0\right] - \widehat{\mathbb{E}}\left[u(t_{n},\cev{X}_n^{n+1})\,|\,N_{\Delta t} = 0\right]\right| \le C\frac{M!}{2^{M}(2M)!} (\Delta t)^{M},
\end{equation}
where the constant $C$ is independent of $M$ and $\Delta t$. Note that the factor $(\Delta t)^{M}$ comes from the $2{M}$-th order differentiation 
of the function $u(t,\cdot)$ with respect to $\xi$ defined in Eq.~\eqref{E0}.


%
To approximate the integral with respect to jump variable $q$ in Eq.~\eqref{int_ejump}, we divide the interaction domain $E$ by equally spaced mesh size $h>0$. 
Using trapezoidal rule in Eq.~\eqref{E2}, we have the bound
\begin{equation}\label{NC_quad}
\left|\mathbb{E}[u(t_n, \cev{V}_{n}^{n+1})\, \big|\, N_{\Delta t} = 1] - \widetilde{\mathbb{E}}[u(t_n, \cev{V}_{n}^{n+1})\, \big|\, N_{\Delta t} = 1]\right|\le C h^2,
\end{equation}
where constant $C$ depends on the second derivative $\partial^{2} u/\partial q^{2}$ and the volume of $E$. 


Combining Eq.~\eqref{GH_quad} and Eq.~\eqref{NC_quad}, $R^{n+1}_8$ in  Eq.~\eqref{error_quad} is bounded by
\begin{equation}\label{e_13}
\begin{aligned}
|R_8^{n+1}| \le & \;\mathbb{P}(N_{\Delta t} = 0) \left|{\mathbb{E}}\left[u(t_n, \cev{X}_{n}^{n+1})\, \big|\, N_{\Delta t} = 0\right] - \widehat{\mathbb{E}}\left[u(t_n, \cev{X}_{n}^{n+1})\, \big|\, N_{\Delta t} = 0\right]\right|\\
&+  \; \mathbb{P}(N_{\Delta t} = 1)\left| {\mathbb{E}}\left[u(t_n, \cev{V}_{n}^{n+1})\, \big|\, N_{\Delta t} = 1\right] - \widetilde{\mathbb{E}}\left[u(t_n, \cev{V}_{n}^{n+1})\, \big|\, N_{\Delta t} = 1\right]\right|\\
 \leq& \; C \left(\frac{M!}{2^{M}(2M)!} (\Delta t)^{M} + \Delta t \,h^2\right),
    \end{aligned}
\end{equation}
where the constant $C$ is independent of $M$, $h$ and $\Delta t$.

\subsubsection{The estimate of \texorpdfstring{$R^{n+1}_{9}$}{lg}}
For the error $R^{n+1}_9$ in Eq.~\eqref{R9} from the piecewise polynomial interpolation, the standard error bound of piecewise cubic Lagrange interpolation ($p=3$) gives
\begin{equation}\label{analy_R9}
\begin{aligned}
    |R_9^{n+1}| \le & \;\mathbb{P}(N_{\Delta t} = 0) \left|\widehat{\mathbb{E}}\left[u(t_n, \cev{X}_{n}^{n+1}) - u^p(t_n, \cev{X}_{n}^{n+1})\, \big|\, N_{\Delta t} = 0\right] \right| \\
 &+ \; \mathbb{P}(N_{\Delta t} = 1) \left|\widetilde{\mathbb{E}}\left[u(t_n, \cev{V}_{n}^{n+1}) - u^p(t_n, \cev{V}_{n}^{n+1})\, \big|\, N_{\Delta t} = 1\right]\right|
\le \; C(\Delta x)^{4},
\end{aligned}
\end{equation}
where constant $C$ is independent with $\Delta x$.

\subsection{The error estimate of Scheme \ref{scheme1}}\label{analy_full}
We combine the estimates of the truncation errors in Section \ref{sec:residual} to obtain an error estimate of Scheme \ref{scheme1}. Denote
\begin{equation}\label{total_e}
    \max_{j= 1, \ldots, J}\left|e^{n+1}(x_j)\right| := \max_{j= 1, \ldots, J} \left|u(t_{n+1},x_j) - u^{n+1}_j\right|,
\end{equation}
for $n = 0, \ldots, N_t-1$, where $u(t_{n+1},x_j)$ is the exact solution and $u_j^{n+1}$ is the nodal approximation obtained by Scheme \ref{scheme1}. 


%

\begin{theorem}\label{thm_full}
Let $\Delta x$ denote the spatial mesh size, $M$ denote the number of Gauss-Hermite quadrature points, $h$ denote the size of the sub-intervals of the trapezoidal rule, and assume the piecewise cubic ($p=3$) Lagrange interpolation applied in Eq.~\eqref{eq_interp_E}.
Then, for sufficiently small $\Delta t$, we have the following error estimate
\begin{equation}\label{eq_thm49}
\max_{j=1, \ldots, J}|e^{n+1}(x_j)| \leq C\left(\Delta t + \frac{(\Delta x)^{4}}{\Delta t} + (\Delta t)^{M-1} + h^2\right).
\end{equation}
\end{theorem}

\begin{proof}

We subtract $u^{n+1}_j$ defined in Scheme \ref{scheme1} from the exact solution $u(t_{n+1},x_j)$ defined in Eq.~\eqref{eq_utn_approx_j} and obtain 
\begin{equation}\label{full_err_analy}
 e^{n+1}(x_j)  =  e^{n+1}_1(x_j) + e^{n+1}_2(x_j)+ e^{n+1}_3(x_j) + \sum_{i=1}^9 R^{n+1}_i,
\end{equation}
where $e^{n+1}_1(x_j)$, $e^{n+1}_2(x_j)$, $e^{n+1}_3(x_j)$ are defined by 
\begin{equation}\label{full_err_analy1}
\begin{aligned}
e^{n+1}_1(x_j) & := \mathbb{P}(N_{\Delta t} = 0)  \widehat{\mathbb{E}}\left[u^p(t_n, \cev{X}_{n}^{n+1}) - u^{n,p}(\cev{X}_{n}^{n+1})\, \big|\, N_{\Delta t} = 0\right], \\[0pt]
e^{n+1}_2(x_j) & :=  \mathbb{P}(N_{\Delta t} = 1)  \widetilde{\mathbb{E}} \left[ u^p(t_{n},\cev{V}^{n+1}_{n})-u^{n,p}(\cev{V}^{n+1}_{n})\, \big|\,N_{\Delta t} = 1\right],\\[0pt]
e^{n+1}_3(x_j) & := \mathbb{E}[(t_{n+1}-\tau_{n} \vee t_{n})]\,g(t_{n+1},x_j,u(t_{n+1},x_j)) - \Delta t g(t_{n+1},x_j,u^{n+1}_j),\\[-0pt]
\end{aligned}
\end{equation}
respectively. 


For $e^{n+1}_1(x_j)$, we introduce an auxiliary function $\hat{e}^{n}(x)$ defined in $\mathcal{D}$ satisfying three properties: (i) $\hat{e}^{n}(x_j) = u(t_{n},x_j) - u^{n}_j$, for $x_j \in \mathcal{S}$, (ii) $\hat{e}^{n}(x)$ is globally non-overshooting, i.e., $|\hat{e}^{n}(x)| \le \max_{j=1,\ldots, J}|\hat{e}^{n}(x_j)|$ for $x\in \mathcal{D}$, and (iii) $\hat{e}^{n}(x) \in \mathcal{C}_b^{4}(\mathcal{D})$.
%
Such smooth function can be constructed using shape-preserving piecewise rational interpolantion \cite{han2018shape,han2021piecewise} or radial basis functions \cite{wu2011new,ahmad2017positivity}. Note that we only need the existence of the function $\hat{e}^{n}(x)$, and do not need to construct $\hat{e}^{n}(x)$ in this error analysis.
When the existence of $\hat{e}^{n}(x)$ is ensured, $u^p(t_n, x) - u^{n,p}(x)$ can be viewed as a piecewise cubic polynomial interpolation ($p=3$) for $\hat{e}^{n}(x)$. Then, we can obtain the following error bound
\begin{equation}\label{append_2}
    \left|u^p(t_n, x) - u^{n,p}(x)- \hat{e}^{n}(x)\right| \le C(\Delta x)^{4}.
\end{equation}
Thus, the error $e^{n+1}_1(x_j)$ in Eq.~\eqref{full_err_analy1} has the bound
\begin{equation}
\begin{aligned}
&\;|e^{n+1}_1(x_j)|\\
=&\; \left|\mathbb{P}(N_{\Delta t} = 0)  \widehat{\mathbb{E}}\left[u^p(t_n, \cev{X}_{n}^{n+1}) - u^{n,p}(\cev{X}_{n}^{n+1})\, \big|\, N_{\Delta t} = 0\right]\right|\\
\le&\;\left| \widehat{\mathbb{E}}\left[u^p(t_n, \cev{X}_{n}^{n+1}) - u^{n,p}(\cev{X}_{n}^{n+1})-\hat{e}^{n}(\cev{X}_{n}^{n+1}) + \hat{e}^{n}(\cev{X}_{n}^{n+1})\, \big|\, N_{\Delta t} = 0\right]\right|\\
\le&\; \left| \widehat{\mathbb{E}}\left[ \hat{e}^{n}(\cev{X}_{n}^{n+1})  \right]\right| +
\left| \widehat{\mathbb{E}}\left[u^p(t_n, \cev{X}_{n}^{n+1}) - u^{n,p}(\cev{X}_{n}^{n+1})-\hat{e}^{n}(\cev{X}_{n}^{n+1}) \big|N_{\Delta t} = 0\right]\right| \\
\le & \; \max_{j=1,\ldots, J}|e^{n}(x_j)| + C(\Delta x)^{4},
\end{aligned}
\end{equation}
where we have $|\hat{e}^{n}(x_j)| = |e^{n}(x_j)|$, for $x_j \in \mathcal{S}$, according to the above definitions of $\hat{e}^{n}(x_j)$ and $e^{n}(x_j)$.

For the error $e^{n+1}_2(x_j)$, we exploit the fact that $\mathbb{P}(N_{\Delta t} = 1) \sim \mathcal{O}(\Delta t)$ to obtain 
\begin{equation}\label{e22}
\begin{aligned}
    |e_2^{n+1}(x_j)| & \le C\Delta t \max_{j=1,\ldots, J}|e^{n}(x_j)| \sum_{j=1}^J \left|\mathbb{E}[\psi_j(\cev{V}^{n+1}_{n})|N_{\Delta t} =1]\right| \\
    & \le 
    C\Delta t \max_{j=1,\ldots, J}|e^{n}(x_j)|,
\end{aligned}
\end{equation}
where the constant $C$ is independent of $\Delta t$.

For the error $e^{n+1}_3(x_j)$, we have
\begin{equation}\label{e33}
\begin{aligned}
|e^{n+1}_3(x_j)| \le & \;\mathbb{E}[\Delta t-(t_{n+1}-\tau_{n} \vee t_{n})] |g(t_{n+1},x_j,u(t_{n+1},x_j))| \\
& \;+ \Delta t |(g(t_{n+1},x_j,u(t_{n+1},x_j)) - g(t_{n+1},x_j,u^{n+1}_j))|\\
\le & \; \mathbb{E}[\Delta t-(t_{n+1}-\tau_{n} \vee t_{n})] |g(t_{n+1},x_j,u(t_{n+1},x_j))| + L\Delta t |e^{n+1}(x_j)|,
\end{aligned}
\end{equation}
where $L$ is the Lipschitz constant of $g$. The expectation $\mathbb{E}[\Delta t-(t_{n+1}-\tau_{n} \vee t_{n})] $ can be estimated by exploiting the fact that $\mathbb{P}(\tau_{n} > t_{n}) \le C \Delta t$ when the spatial mesh $\mathcal{S}$ satisfies the condition in Eq.~\eqref{cond_bound1}, i.e., 
\[
\begin{aligned}
\mathbb{E}[\Delta t-(t_{n+1}-\tau_{n} \vee t_{n})] = \mathbb{P}(\tau_{n} > t_{n})\mathbb{E}[(t_{n+1}-\tau_{n} )] \le C(\Delta t)^2.
\end{aligned}
\]
Substituting the above estimate into Eq.~\eqref{e33}, we have
\begin{equation}\label{analy_e2}
|e^{n+1}_3(x_j)| 
\le L\Delta t |e^{n+1}(x_j)| + C(\Delta t)^2,
\end{equation}
where $L$ is the Lipschitz constant of $g$.

Now we substitute the estimates of truncation errors $R^{n+1}_i$, for $i = 1,\ldots, 9$, obtained in Section \ref{sec:residual} and the estimates of $e^{n+1}_1(x_j)$, $e^{n+1}_2(x_j)$, $e^{n+1}_3(x_j)$ into Eq.~\eqref{full_err_analy} to obtain 
%
%
\begin{equation}
\begin{aligned}
|e^{n+1}(x_j)| \leq& \;\max_{j=1, \ldots, J}|e^{n}(x_j)|
\big(1 + C\Delta t\big)+L\Delta t |e^{n+1}(x_j)|\\
& \; + C\left( (\Delta x)^{4}  + (\Delta t)^2 + (\Delta t)^{M}+ \Delta t h^2\right).
\end{aligned}
\end{equation}
Then we have
\begin{equation}
\begin{aligned}
(1-C\Delta t)\max_{j=1, \ldots, J}|e^{n+1}(x_j)| \leq & (1+C\Delta t) \max_{j =1, \ldots, J}|e^{n}(x_j)|\\
& + C\left( (\Delta x)^{4}  + (\Delta t)^2 +  (\Delta t)^{M} + \Delta t \,h^2\right). \\
\end{aligned}
\end{equation}
From the above, we obtain, for sufficiently small $\Delta t$, 
\begin{equation}\label{analyerr_nodal}
\max_{j=1, \ldots, J}|e^{n+1}(x_j)| \leq C\left(\Delta t + \frac{(\Delta x)^{4}}{\Delta t} +  (\Delta t)^{M-1} + h^2\right),
\end{equation}
where $C$ is a constant independent of $\Delta t$.
\end{proof}

\section{Numerical examples}\label{sec_ex}
In this section we present two numerical examples to demonstrate the performance of the proposed method. Specifically, the example in Section \ref{ex_1} aims at verifying the convergence rate of Scheme \ref{scheme1} proved in Theorem \ref{thm_full}, and the example in Section \ref{ex_2} is to illustrate the application of the proposed method to a problem motivated by the study of heat transport in magnetically confined controlled nuclear fusion plasmas.



\subsection{3D nonlocal diffusion with volume constraints in irregular and bounded domains}\label{ex_1}
We consider the following nonlocal diffusion equation
\begin{equation}\label{pide_ex1}
\begin{array}{lr}
\begin{aligned}
\frac{\partial u}{\partial t}(t,x) - \mathcal{L}[u](t,x) - f(t,x,u)&=0, \quad\quad \;\;\;\quad \forall (t,x) \in (0,T] \times \mathcal{D},\\[0pt]
u(0,x) &= \phi_0(x),\quad\quad  \forall x\in \mathcal{D}\cup \mathcal{D}_{\rm v},\\[4pt]
u(t,x) &= \phi_{\rm v}(t,x), \quad\; \forall (t,x) \in (0,T] \times \mathcal{D}_{\rm v},
\end{aligned}
\end{array}
\end{equation}
with $x := (x_1, x_2, x_3) \in \mathbb{R}^3$, and $\mathcal{L}$ the operator  in Eq.~\eqref{fpon} with coefficients:
\vspace{0.1cm}
\begin{itemize}[leftmargin=20pt]\itemsep0.15cm
    \item Drift term $B(t,x) = t\begin{pmatrix}x_1^5 - \frac{5}{3}x_1^3,\; x_2^5 - \frac{5}{3}x_2^3,\; x_3^5 - \frac{5}{3}x_3^3\end{pmatrix}$
    \item Jump amplitude $c(t,x,q) = \begin{pmatrix}q_1,\; q_2 + \frac{t}{2},\; q_3 +\frac{1}{4}x_3\end{pmatrix}$
    \item Kernel $\gamma(q) = {\bf 1}_{|q|\leq \delta}/\delta^3$
    \item Diffusion coefficient $\sigma(t,x) = (1-t)
    \tiny
    \begin{pmatrix}\sin(x_1)& 0 & 0\\
    0 & \cos(x_2) & 0 \\
    0 & 0 & x_3\\
    \end{pmatrix}$
\end{itemize}
Following the method of manufactured solutions,
we choose 
\begin{equation}\label{ex1_exact}
    u(t,x) = \sin(5t)( x_1^4-x_1^2 + x_2^4-x_2^2 + x_3^4-x_3^2),
\end{equation}
which determines the initial condition $\phi_0(x)$, and the volume constraint $\phi_{\rm v}(t,x)$,
and construct the nonlinear forcing term $f(t,x,u)$ as
\begin{equation}
\begin{aligned}
f &\; = \; 5\cos(5t)( x_1^4-x_1^2 + x_2^4-x_2^2 + x_3^4-x_3^2)\\
 &\;- u^2(t,x) - \sin(5t)\left( \sum_{i = 1}^3 (4x_i^3-2x_i)B_i(t,x)\right)\\
&\; - u(t,x)\Big(2\cos(2x_1) - 2\cos(2x_2) + 2\Big) \\
&\;- \sin(5t)\Big(2(4x_1^3 - 2x_1)\sin(2x_1)- 2(4x_2^3 - 2x_2)\sin(2x_2) + 16x_3^4 - 8x_3^2\Big)\\ &\;- \sin(5t)\Big((12x_1^2 - 2)\sin^2(x_1) + (12x_2^2 - 2)\cos^2(x_1) + 12x_3^4 - 2x_3^2\Big)\\
&\;-\sin(5t) \bigg( \frac{12\pi}{35}\delta^4 + \frac{4\pi}{15}\delta^2 \Big(6x_1^2 + 6x_2t + 6x_2^2 + \frac{3}{2}t^2 + \frac{75}{8}x_3^2 - 1\Big)\\
&\;+\frac{4\pi}{3}\Big(2x_2^3t + \frac{1}{2}x_2t^3 + \frac{3}{2}x_2^2t^2 + \frac{t^4}{16} + \frac{77}{64}x_3^4\Big)
\bigg),
\end{aligned}
\end{equation}
to guaranteed that $u$ is an exact  solution of Eq.~(\ref{pide_ex1}).
%

 The interaction domain $\mathcal{D}_{\rm v}$ is defined by the extension from $\mathcal{D}$ by a radius of the horizon $\delta$. 
We set the terminal time $T = 0.5$, $\delta = 0.3$ and solve Eq.~\eqref{pide_ex1} on the cubic domain $[0,1]^3$. We use piecewise cubic Lagrange interpolation to approximate $u(t,x)$ in $\mathcal{D}$ in Eq.~\eqref{eq_interp_E}, and use the trapezoidal quadrature rule  to approximate the conditional expectation in Eq.~\eqref{E2}.
The goal of this example is to demonstrate Scheme \ref{scheme1} can achieve first-order convergence with respect to $\Delta t$ when we use the error estimate in Theorem \ref{thm_full} to choose
the spatial mesh size $\Delta x$, the number of Gauss-Hermite quadrature points $M$ and mesh size $h$ for the Newton-Cotes quadrature rule. According to the error bound in Eq.~\eqref{eq_thm49}, we set $M = 2$, $h \sim (\Delta t)^{\frac{1}{2}}$, and $\Delta x \sim (\Delta t)^{\frac{1}{2}}$ to achieve the first-order convergence with respect to $\mathcal{O}(\Delta t)$.
%

%

Table \ref{tb_1} demonstrates how $\Delta x$ affects the convergence rate while keeping $M = 2$ and $h \sim (\Delta t)^{\frac{1}{2}}$. When choosing $\Delta x$ guided by Theorem \ref{thm_full}, Scheme \ref{scheme1} achieves the desired $\mathcal{O}(\Delta t)$ convergence rate. When enlarging $\Delta x$ to $(\Delta t)^{1/3}$, Theorem \ref{thm_full} suggests that the total error is dominated by the term $(\Delta x)^4/\Delta t = (\Delta t)^{1/3}$ in Eq.\eqref{eq_thm49}. In the experiment, we obtain 0.3641 convergence rate that is very close to the 1/3-order theoretical convergence rate. This indicates the tightness of our error bound. On the other hand, when reducing $\Delta x$ to $\Delta t$, we observe that Scheme \ref{scheme1} only achieves half-order convergence. This is due to the violation of the condition in Eq.~\eqref{cond_bound1}, such that the error caused by neglecting the truncation errors $R^{n+1}_2$ in Eq.~\eqref{e100} and $R^{n+1}_3$ in Eq.~\eqref{e101} in Scheme \ref{scheme1} becomes dominant.

Table \ref{tb_quadrature} demonstrates the influence of the number of the Gauss-Hermite quadrature points $M$ on the $L^2$ error and the convergence rate with respect to $\Delta t$ while keeping $h \sim {(\Delta t)^{1/2}}$, $\Delta x \sim {(\Delta t)^{1/2}}$. When using only one Gauss-Hermite quadrature point, i.e., $M=1$, it is equivalent to completely neglecting the local diffusion in Eq.~\eqref{pide_ex1}. Then it is expected that Scheme \ref{scheme1} cannot converge. On the other hand, using three quadrature points does not improve the convergence rate, which verifies the correctness of the error bound in Theorem \ref{thm_full}.

Table \ref{tb_quadrature1} demonstrates the influence of mesh size $h$ of the trapezoidal rule on the $L^2$ error and the convergence rate with respect to $\Delta t$ while keeping $\Delta x \sim {(\Delta t)^{1/2}}$, $M = 2$. As expected, enlarging $h$ to $(\Delta t)^{1/4}$ reduces the convergence rate to half order, which is consistent with the error bound in Eq.~\eqref{eq_thm49}.

%

\begin{table}[h!]
\footnotesize
\renewcommand{\arraystretch}{1.3}
\centering
\caption{Demonstration of the influence of $\Delta x$ on the $L^2$ error and the convergence rate (CR) with respect to $\Delta t$ while keeping $h \sim {(\Delta t)^{1/2}}$, $M = 2$, for the example in Section \ref{pide_ex1}.}\label{tb_1}
\vspace{-0.2cm}
  \begin{tabular}{c c c c c c c}
  \hline
  $\Delta t$ & $2^{-6}$ & $2^{-7}$ & $2^{-8}$ & $2^{-9}$ & $2^{-10}$& CR\\
  \hline
     $\Delta x \sim (\Delta t)^{\frac{1}{2}}$    & 1.3800e-02 &6.9718e-03& 3.4739e-03 &1.7379e-03 &8.9337e-04   & 0.9903\\
     $\Delta x \sim (\Delta t)^{\frac{1}{3}}$  &   1.4317e-02 &  1.0397e-02  & 7.6536e-03 &  6.4154e-03  & 5.1604e-03 & 0.3641\\
     $\Delta x \sim \Delta t$  & 2.0540e-02 &  1.4316e-02  & 1.0091e-02  & 7.1694e-03 &  5.1054e-03  & 0.5014 \\
  \hline
  \end{tabular}
  \vspace{-0.0cm}
\end{table}
\begin{table}[h!]
\footnotesize
\renewcommand{\arraystretch}{1.3}
\centering
 \caption{Demonstration of the influence of the number of the Gauss-Hermite quadrature points $M$ on the $L^2$ error and the convergence rate (CR) with respect to $\Delta t$ while keeping $h \sim {(\Delta t)^{1/2}}$, $\Delta x \sim {(\Delta t)^{1/2}}$, for the example in Section \ref{pide_ex1}.}
 \label{tb_quadrature}
 \vspace{-0.2cm}
  \begin{tabular}{c c c c c c c}
   \hline
 $\Delta t$ & $2^{-6}$ & $2^{-7}$ & $2^{-8}$ & $2^{-9}$ & $2^{-10}$& CR\\
   \hline
     $M = 1$ &  3.9569e-02 & 5.3194e-02  & 6.4850e-02  & 7.7741e-02 & 8.8012e-02 & -0.2854  \\
     \hline
      $M = 2$  & 1.3800e-02 &6.9718e-03& 3.4739e-03 &1.7379e-03 &8.9337e-04   & 0.9903 \\
   \hline
    $M = 3$ &1.3111e-02& 6.5217e-03& 3.1162e-03 & 1.5971e-03 &7.8121e-04     & 1.0168  \\
   \hline
  \end{tabular}
\end{table}
\begin{table}[h!]
\footnotesize
\renewcommand{\arraystretch}{1.3}
\centering
 \caption{Demonstration of the influence of mesh size $h$ of the trapezoidal rule on the $L^2$ error and the convergence rate (CR) with respect to $\Delta t$ while keeping $\Delta x \sim {(\Delta t)^{1/2}}$, $M = 2$, for the example in Section \ref{pide_ex1}.}
 \label{tb_quadrature1}
 \vspace{-0.2cm}
  \begin{tabular}{c c c c c c c}
     \hline
 $\Delta t$ & $2^{-6}$ & $2^{-7}$ & $2^{-8}$ & $2^{-9}$ & $2^{-10}$& CR\\
   \hline
   $h \sim \Delta t$ &1.3145e-02 & 6.2181e-03 &  3.1537e-03 &1.6222e-03& 7.8090e-04  & 1.0085\\
   \hline 
     $h \sim (\Delta t)^{\frac{1}{2}}$  & 1.3800e-02 &6.9718e-03& 3.4739e-03 &1.7379e-03 &8.9337e-04   & 0.9903 \\
   \hline   
   $h \sim (\Delta t)^{\frac{1}{4}}$   &6.8339e-02 & 4.9122e-02  & 3.1491e-02  & 2.0533e-02 & 1.2051e-02 & 0.4889\\
   \hline
  \end{tabular}
  \vspace{-0.1cm}
\end{table}

Next, we test the performance of Scheme \ref{scheme1} by solving Eq.~\eqref{pide_ex1} in the four domains of different shapes, shown in Figure \ref{fig_ex1_domain}, in order to demonstrate the broad applicability of our method. The interaction domain $\mathcal{D}_{\rm v}$ is defined by the extension from $\mathcal{D}$ by a radius of the horizon $\delta$. The tetrahedral meshes are generated using DistMesh code \cite{persson2004simple} with the maximum mesh size being 0.025. We set $T = 0.5$, $h \sim (\Delta t)^{1/2}$ guided by Theorem \ref{thm_full}.
The result are shown in Table \ref{tb_2}. As expected, we observe the first-order convergence with respect to $\Delta t$ in all the four cases. 
\begin{figure}[h!]
    \centering
  {\includegraphics[width=\textwidth]{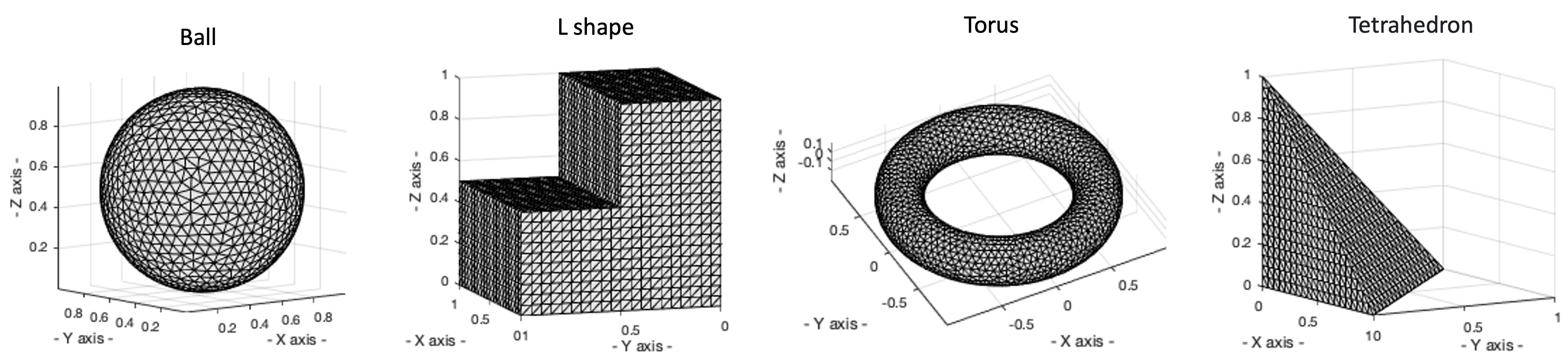}}
  \vspace{-0.7cm}
    \caption{The four domains of different shapes used to test the performance of Scheme \ref{scheme1} by solving Eq.~\eqref{pide_ex1}. The tetrahedral meshes for the four domains were generated with the maximum mesh size being 0.025.}
    \label{fig_ex1_domain}
    \vspace{-0.3cm}
\end{figure}

\begin{table}[h!]
\footnotesize
\centering
\renewcommand{\arraystretch}{1.25}
 \caption{The $L^2$ errors and the convergence rates with respect to $\Delta t$ for solving Eq.~\eqref{pide_ex1} in the four domains in Figure \ref{fig_ex1_domain}, where $T = 0.5$, $h \sim \sqrt{\Delta t}$, $M = 2$.}
 \label{tb_2}
 \vspace{-0.2cm}
  \begin{tabular}{c c c c c c}
   \hline
  $\Delta t$ & $0.1$ & $0.05$ & $0.025$ & $0.0125$ & CR\\
   \hline
    Ball  & 0.0501&  0.0228 & 0.0109     & 0.0051  &  1.0953 \\
    L shape & 0.0533 & 0.0296  &  0.0122&  0.0061   & 1.0660\\
    Torus  & 0.0301 & 0.0169   & 0.0081&   0.0041     & 0.9689 \\
    Tetrahedron & 0.0314 &  0.0123& 0.0075     & 0.0037   & 0.9969  \\
   \hline
  \end{tabular}
\end{table}

\subsection{Heat transport in magnetically confined plasma for controlled nuclear fusion}\label{ex_2}
This example is motivated by the study of heat transport in magnetically confined plasmas. The most promising approach to achieve controlled nuclear fusion for energy production is to heat a plasma composed of hydrogen isotopes at high enough temperature, high enough density and for a long enough time. Among the many  complex physical processes that need to be understood to achieve this, the transport of particles and heat play a key role. In particular, if the heat losses are too high the plasma will not reach the required temperature for nuclear fusion. Transport theories based on locality assumptions lead to the well understood advection-diffusion partial differential equations models. However recent studies have cast doubts on these simple models due to the role played by non-local transport. In particular, in a magnetized plasma transport is highly anisotropic: parallel (along the magnetic field) transport can be nonlocal, while perpendicular transport is usually local. As a simplified model to study the role of this local/nonlocal anisotropy we consider the following model 
\begin{equation}\label{pide_ex3}
\frac{\partial u}{\partial t}(t,\phi,\theta,r) - \mathcal{L}[u](t,\phi,\theta,r) =0,
\end{equation} 
where $u$ represents the scalar filed being transported, e.g., temperature, and the operator ${L}$ is given by
\begin{equation}
\mathcal{L}[u](t,\phi,\theta) =  \int_{|\bf{q}|\leq \pi}[u(t,\phi + q_1, \theta +q_2,r) - u(t,\phi,\theta,r)] \gamma(\hat{q}_1,\hat{q}_2)d{\bf q} + \frac{1}{2}\sigma^2 \frac{\partial^2 u}{\partial r^2}.
\label{operator}
\end{equation}

\begin{figure}[h!]
    \centering
  {\includegraphics[scale = 0.3]{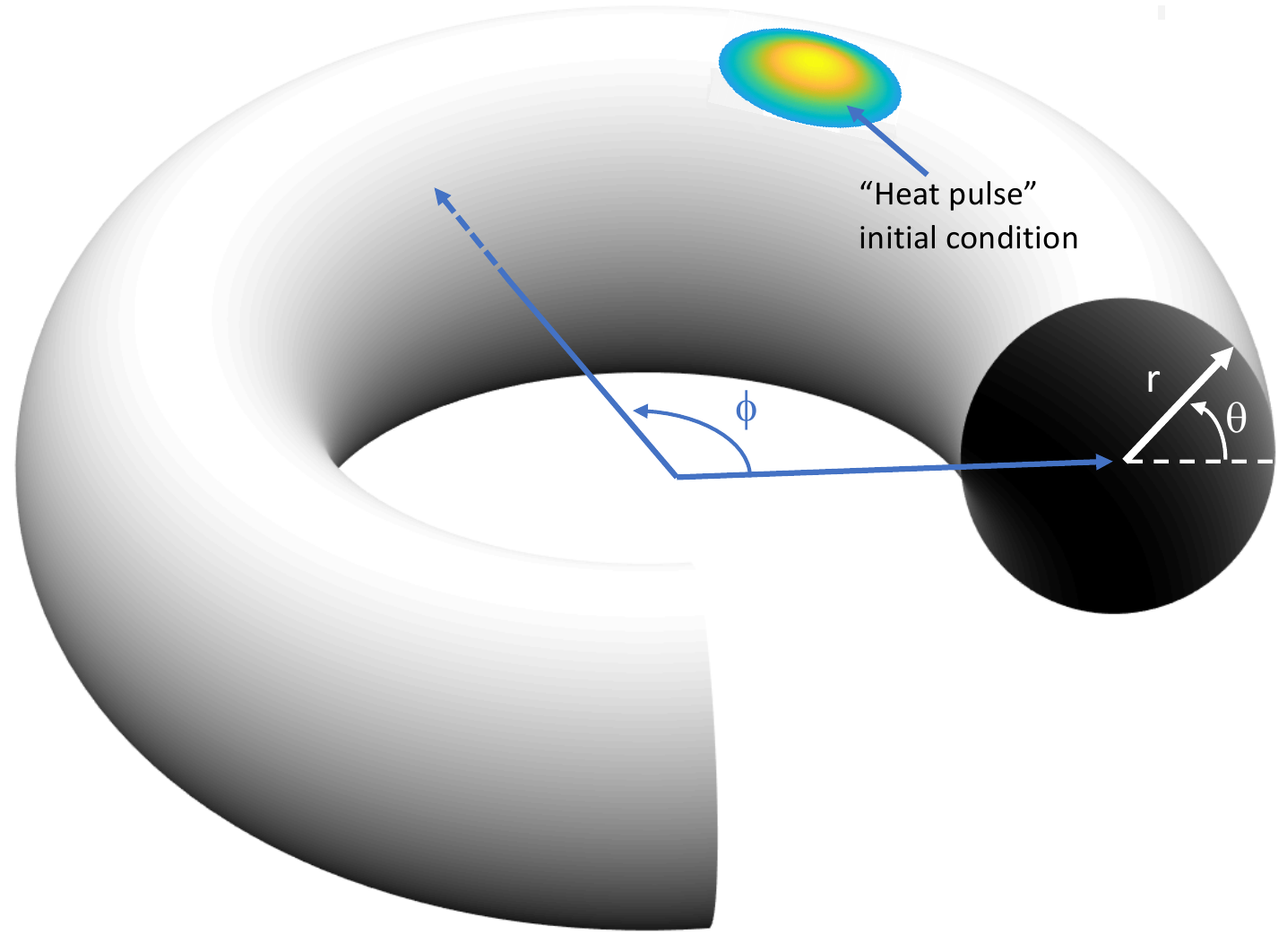} }
    \caption{Schematic representation of ``heat pulse" propagation problem in 3D toroidal geometry. The toroidal domain is parameterized by the poloidal, $0 \leq \theta <2 \pi$, and toroidal, $0 \leq \phi < 2\pi$, angles and the minor radius $0<r \leq 0.5$. The initial condition in Eq.\eqref{eq_ex2_ini} corresponds to a ``heat pulse" localized at $(r,\phi, \theta)=(r_0,\phi_0, \theta_0)$. The problem is to compute the spatiotemporal evolution of the  ``heat pulse" inside the torus by solving the nonlocal transport equation in Eq.\eqref{pide_ex3} with boundary conditions in Eqs.\eqref{eq_ex2_b1} and \eqref{eq_ex2_b2}. This problem is motivated by the study of heat transport in magnetically confined plasmas in controlled nuclear fusion. 
    }
    \label{torus}
\end{figure}
The domain of interest is the 3D torus shown in Fig.~\ref{torus} where $0\leq \phi < 2 \pi$ and 
$0\leq \theta < 2 \pi$ are the toroidal and poloidal angles and $0<r \leq 0.5$ is the minor radius.
The boundary conditions are double periodic in $\phi$ and $\theta$
\begin{equation}\label{eq_ex2_b1}
 u(t,\phi,\theta,r)=u(t,\phi+2 \pi,\theta+ 2\pi, r)\, ,
\end{equation}
and
\begin{equation}\label{eq_ex2_b2}
 \partial_r u(t,\phi,\theta,r=0)=0 \qquad  
 u(t,\phi,\theta,r=0.5)=0 \, .
\end{equation}
According to the last term on the right hand side of Eq.~(\ref{operator}) transport in the radial direction is assumed to be locally diffusive. On the other hand, transport in the $\phi$ and $\theta$ directions is nonlocal and governed by the kernel 
\[
\gamma(\hat{q}_1,\hat{q}_2) = \frac{e^{\kappa_1 \cos{\hat{q}_1}}e^{\kappa_2 \cos{\hat{q}_2}}}{\int_{|\bf{\hat{q}}|\leq \pi}\gamma(\hat{q}_1,\hat{q}_2)d\bf{\hat{q}}} , 
\]
that involves truncated von Mises probability density functions.
In this model the strength of the non-locality in the angular variables $\phi$ and $\theta$ is determined by the parameters $\kappa_1$ and $\kappa_2$ respectively. In particular, the smaller the value of $\kappa_i$ the  stronger the non-locality in the corresponding direction. In fusion plasmas, the magnetic field winds over the toroidal surfaces and as a result the direction of stronger non-locality is not aligned with the $\phi$ or $\theta$ direction. To incorporate this important aspect in the model we define
\begin{equation}\label{eq_rota}
    \begin{aligned}
\hat{q}_1 =& q_1 \cos{\psi} + q_2 \sin{\psi},\\
\hat{q}_2 =& -q_1 \sin{\psi} + q_2 \cos{\psi}.
\end{aligned}
\end{equation}
where the angle $\psi$ determines the direction of maximum non-locality. 
The initial condition corresponds to a ``heat pulse" represented by a Gaussian distribution centered at $(\phi_0,\theta_0,r_0)$
\begin{equation}\label{eq_ex2_ini}
    u(0,\phi,\theta,r) = \exp\Big(-\frac{(\phi-\phi_0)^2}{0.5}-\frac{(\theta-\theta_0)^2}{0.5}-\frac{(r-r_0)^2}{0.005}\Big)
    \end{equation}

In this numerical experiment we use $\sigma = 0.01$,
$\psi=30^{\circ}$, $\kappa_1=20$, $r_0=0.25$, $\phi_0=\pi$ and $\theta_0=\pi$. To explore the role of different levels of nonlocality we will consider the  following values of $\kappa_2=20\, , 10\, , 5  \, , 2 \, , 1$ and $0.1$. The maximum integration time will be $t=32$, and 
to visualize the results, the value of $u$ on a given torus with a fixed $r$ at a time $t$, will be represented on the  double periodic  $[0, 2\pi) \times [0, 2\pi)$ Cartesian plane $(\phi,\theta)$.  
%
%
\begin{figure}[htbp]
    \centering
  \includegraphics[width=0.8\textwidth]{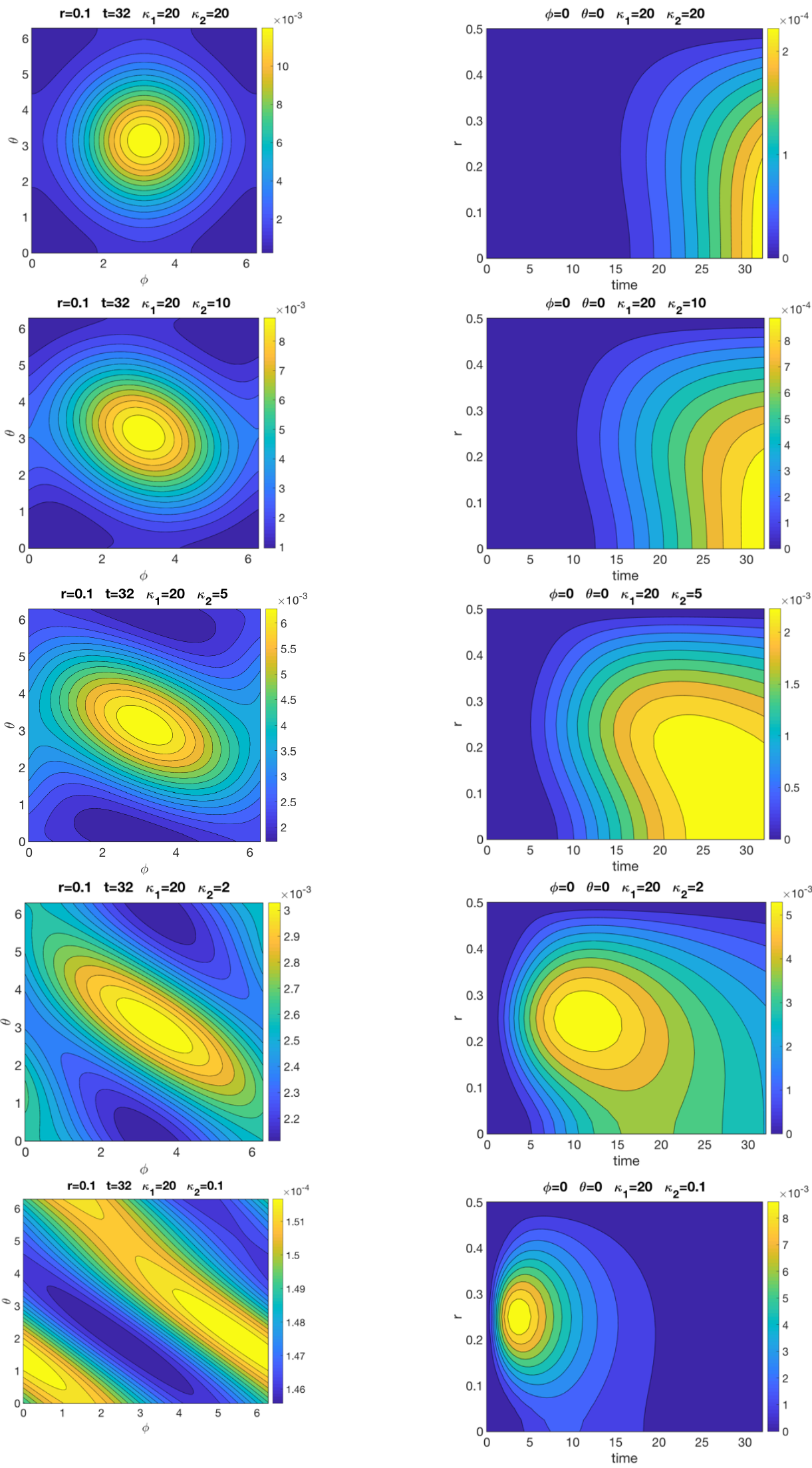}
  \vspace{-0.4cm}
    \caption{Spatiotemporal evolution of heat pulse for different levels of non-locality. The left column shows contour plots of $u$ at $t=32$ (final time) in the $\phi \times \theta$ double periodic plane (torus) at the fixed minor radius $r=0.1$. The right column shows the corresponding contour plots of the time evolution of the radial profile of $u$ at $\phi=\theta=0$. In all cases $\kappa_1=20$ while the value of $\kappa_2$ is changed from weak non-locality, $\kappa_2=20$, to strong non-locality, $\kappa_2=1$. The level of anisotropy was kept fixed at $\psi=30^\circ$. 
    }
    \label{cases}
    \vspace{-0.6cm}
\end{figure}

Figure~\ref{cases} shows the spatiotemporal evolution of $u$. As indicated before, the initial pulse is centered at $r=0.25$.  The plots on the left column of Fig.~\ref{cases} show contour plots of $u$ at the final time, $t=32$, in the $(\phi,\theta)$ double periodic Cartesian plane at the inner radius $r=0.1$, for different levels of no-locality.
It is observed that as  $\kappa_2$ is decreased, the non-locality gives rise to a stronger mixing and eventual filamentation of the initial Gaussian pulse with a tilt determined by the anisotropy direction parameter $\psi$.  Note also that this mxing in the $(\phi,\theta)$ plane is accompanied by a reduction of the peak value of $u$. The accompanying plots on the right column of  Fig.~\ref{cases} show the radial profiles of the response in time  at $(\phi,\theta)=(0,0)$, a location opposite to  where the initial pulse was introduced,  $(\phi,\theta)=(\pi,\pi)$. It is observed that, as the nonlocality increases, the response is faster and the peak of the response approaches $r=0.25$.

\begin{figure}[h!]
    \centering
  {\includegraphics[scale = 0.5]{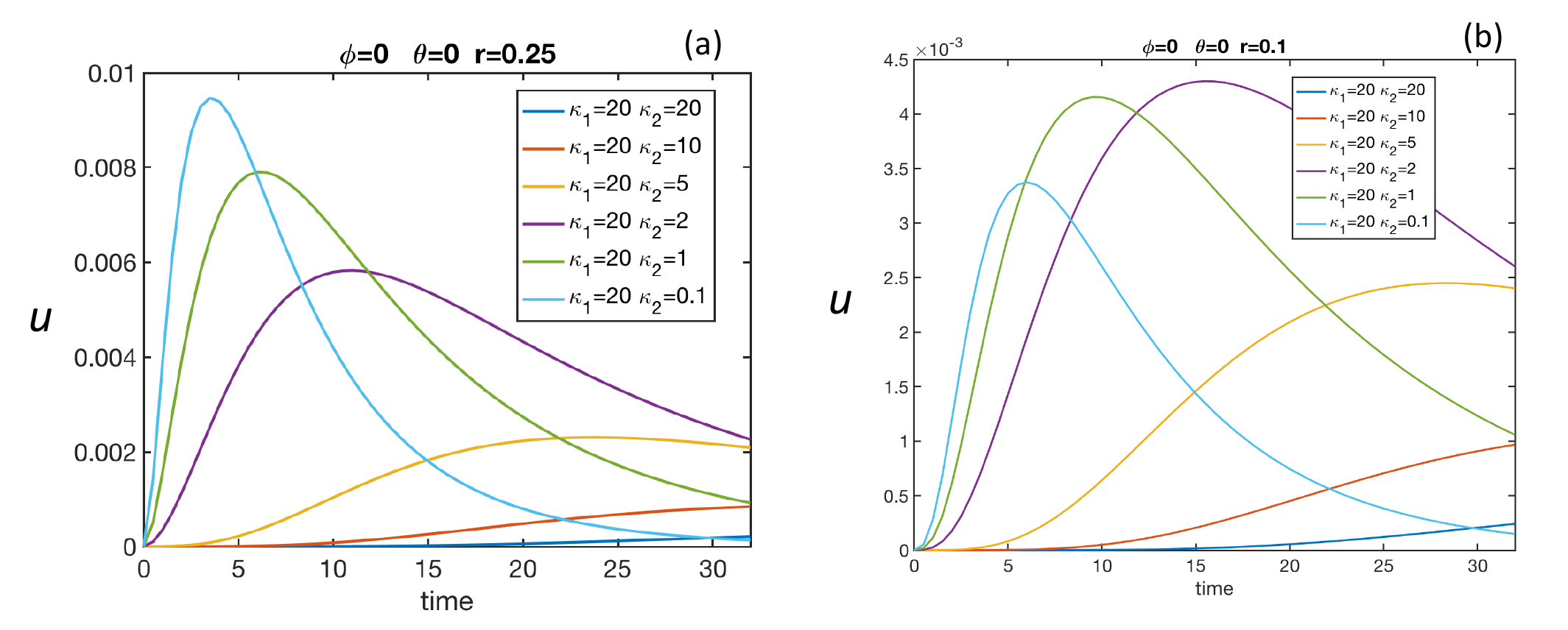} }
  \vspace{-0.4cm}
    \caption{Heat pulse response at $\phi=\theta=0$ and minor radius $r=0.25$, panel (a), and $r=0.1$, panel (b), for different levels of non-locality. }
    \label{pulses}
\end{figure}
Figure~\ref{pulses} shows the response curves for different levels of non-locality at two different locations: (a) $(\phi,\theta,r)=(0,0,0.25)$ which corresponds to the same torus where the initial pulse is introduced and at (b)  $(\phi,\theta,r)=(0,0,0.1)$ which corresponds to an inner torus. In the first case, the time of peaking and the magnitude of the peaking of the response curves is directly proportional to the level of non-locality. However, in the second case that involves the nonlocal propagation of the perturbations in the angle and the radial diffusion from the radius where the pulse is introduced, $r=0.25$, to the point of observation, $r=0.1$, the response curves show a more complicated dependence on $\kappa_2$. Understanding the dependence of the response curves on the nonlocality is key in the experimental characterization of transport in plasma physics, as well as geophysics and fluid dynamics in general. In fact, in fusion experiments the response of the plasma to ``cold" pulse perturbation is used to assess the possible existence of non-local transport, and to validate and calibrate models. An example of this, is the work on Ref.~\cite{del2008fractional} where non-local fractional transport models were used to interpret  experimental results on heat pulse propagation in the JET (Joint European Torus) tokamak fusion experiment.
The proposed transport model, as well as the numerical method, open the possibility of performing nonlocal transport simulations in fusion plasmas incorporating 3D effects and general nonlocal kernels. This type of numerical experiments are valuable to develop, calibrate, and validate predictive model of plasma transport. 

\section{Conclusion}\label{sec:con}
We developed a novel probabilistic scheme for a class of time-dependent semilinear nonlocal diffusion equations with volume constraints and nonlinear forcing. Rigorous error estimates of the proposed fully discrete method were given to demonstrate the first-order convergence with respect to time step size $\Delta t$. We presented two numerical examples illustrating specific aspects and advantages of the proposed numerical method. The first example showed our method's superior performance on 3D semilinear nonlocal diffusion problems in non-trivial domains. The theoretical results were numerically verified in this example. The second example considered an anisotropic nonlocal heat transport problem of interest to magnetically confined controlled nuclear fusion plasmas and illustrated the capability of the proposed method for handling complex physics problems.

We limited attention to  semilinear nonlocal diffusion equations with integrable kernels. 
Our next step is to extend the current scheme to enable its use in non-integrable kernels, e.g., the fractional Laplacian, which requires different discretization schemes for the corresponding stochastic processes and new quadrature rules for estimating the resulting conditional expectations.
Moreover, the current scheme does not include the capability of adaptive spatial mesh refinement to handle the scenario of having non-smooth or even discontinuous solutions. Since our numerical method does not require solving linear systems, it would be fairly easy to add a mesh refinement strategy to Steps 4 and 5 in Scheme \ref{scheme1}. Lastly, in more complex problems, the Euler scheme in Eq.~\eqref{Euler} is too simple to describe the spatio-temporal evolution of particles (electrons). This task, which is quite challenging in the context of PDE-based methods, can be accomplished by replacing the Euler scheme  with the temporal propagators provided by the external particle simulator.



    
    

\section*{Acknowledgments}
This material is based upon work supported in part by the U.S. Department of Energy, Office of Science, Office of Advanced Scientific Computing Research and Fusion Energy Science, and by the Laboratory Directed Research and Development program at the Oak Ridge National Laboratory, which is operated by UT-Battelle, LLC, for the U.S.~Department of Energy under Contract DE-AC05-00OR22725.

\bibliographystyle{siamplain}
\bibliography{PIDE_ref,nonlocal_ref}

\end{document}